\documentclass[11pt,reqno]{amsart}
\usepackage[T1]{fontenc}

\usepackage{color}
\definecolor{MyLinkColor}{rgb}{0,0,0.4}

\newcommand{\R}{{\mathbb R}}
\newcommand{\E}{{\mathcal E}}

\newcommand{\N}{{\mathbb N}}

\newcommand{\cF}{\mathcal{F}}

\newcommand{\cK}{\mathcal{K}}

\newcommand{\D}{\mathbb{D}}

\newcommand{\rhp}{\rightharpoonup}

\newcommand{\p}{\partial}

\newcommand{\e}{\varepsilon}

\newcommand{\id}{\mathop{\rm id}\nolimits}

\newcommand{\J}{\mathop{\rm J}\nolimits}

\newtheorem{thm}{Theorem}[section]

\newtheorem{lemma}[thm]{Lemma}

\theoremstyle{remark} 
\newtheorem{rem}[thm]{Remark}
\numberwithin{equation}{section}

\title[A thin film approximation of the Muskat problem]{A thin film approximation of the Muskat problem with gravity and capillary forces}
\thanks{}

\author[Ph. Lauren\c cot]{Philippe Lauren\c cot}
\address{Institut  de Math\'ematiques de Toulouse, CNRS UMR 5219, Universit\'e de Toulouse, F-31062  Toulouse cedex 9, France}
\email{laurenco@math.univ-toulouse.fr}

\author[B.--V. Matioc]{Bogdan--Vasile Matioc}
\address{Institut f\" ur Mathematik, Universit\" at Wien, Nordbergstra{\ss}e 15, 1090 Wien, {\"O}sterreich}
\email{bogdan-vasile.matioc@univie.ac.at}

\date{\today}


\usepackage[colorlinks=true,linkcolor=MyLinkColor,citecolor=MyLinkColor]{hyperref} 

\begin{document}

\begin{abstract}
Existence of nonnegative weak solutions is shown for a thin film approximation of the Muskat problem with gravity and capillary forces taken into account. The model describes the space-time evolution of the heights of the two fluid layers and is a fully coupled system of two fourth order degenerate parabolic equations. The existence proof relies on the fact that this system can be viewed as a gradient flow for the $2-$Wasserstein distance in the space of probability measures with finite second moment.
\end{abstract}

\maketitle

\section{Introduction and main result}\label{S:0}
The Muskat problem is a free boundary problem describing the motion of two immiscible fluids with different densities and viscosities in a porous medium (such as intrusion of water into oil). 
It gives the space and time evolution of the heights $f\ge 0$ and $g\ge 0$ of the two fluid layers, with the first layer, of height $f$, located on a impermeable horizontal bottom and the second one, of height $g$, 
on top of it, as well as that of the pressure fields inside the fluids. As it involves four unknowns and two free boundaries, one separating the lower and the upper fluid  and one separating the upper fluid and the air, it is a complex problem. 
Recently, using a lubrication approximation, see, e.g., \cite[Chapter~5.B]{Le07}, a thin film approximation to the Muskat problem has been derived in \cite{EMM12} which retains only the
 heights $f$ and $g$ of the two fluid layers as unknowns and reads
\begin{subequations}\label{eq:problem}
\begin{equation}\label{Prob}
\left\{
\begin{aligned}
\p_t f=&\mathrm{div}\left[f\left(-A \nabla\Delta f-B \nabla\Delta g +a\nabla f+b \nabla g\right)\right],\\[1ex]
\p_t g=&\mathrm{div}\left[g\left(-\nabla\Delta f-\nabla\Delta g+c \nabla f+c \nabla g\right)\right],
\end{aligned}\right.\qquad (t,x)\in (0,\infty)\times \R^2,
\end{equation} 
together with initial conditions
\begin{equation}\label{eq:bc1}
 (f,g)(0)=(f_0,g_0), \qquad x\in \R^2.
\end{equation}
\end{subequations}
The parameters $(A,B,a,b,c)$ involved in \eqref{Prob} depend on the densities, viscosities, and surface tensions of the fluids and are assumed to satisfy
\begin{equation}\label{constants}
(a, b, c)\in[0,\infty)^3, \qquad cB=b,\qquad\text{and} \qquad A>B>0.
\end{equation}
Let us recall that the second order terms in \eqref{Prob} account for gravity forces while the fourth order terms result from capillary forces in the original Muskat problem.

The problem \eqref{Prob} is a fourth order degenerate parabolic system with a full diffusion matrix. Its parabolicity has been exploited in \cite{EMxx} to show 
the local existence and uniqueness of positive strong solutions to \eqref{Prob} in a bounded interval $(0,L)$ with no slip boundary conditions. 
Global existence for initial data close to a positive flat steady state is also proved when $\pi^2(A-B)/L^2 + (a-b)>0$ as well as the stability of this steady state. 
The existence and stability of flat and non-flat stationary solutions are also discussed according to the values of the parameters. 
Nonnegative global solutions have also been constructed either when only gravity is taken into account ($A=B=0$) and \eqref{eq:problem} is considered in a bounded interval with homogeneous 
Neumann boundary conditions \cite{ELM11} or in $\R$ \cite{LMxx}, or when gravity forces are discarded ($a=b=c=0$) and \eqref{eq:problem} is considered in a bounded interval with no slip boundary conditions \cite{Mxx}. 
An important tool in the above mentioned works is the availability of two Liapunov functionals for \eqref{eq:problem}, one being the functional $\E$ defined by 
\begin{equation}\label{E}
\E(f,g):=\frac{1}{2}\int_{\R^2} \left[ (A-B)|\nabla f|^2+B|\nabla (f+ g)|^2+(a-b)f^2+b(f+g)^2 \right]\, dx
\end{equation}
for $(f,g)\in H^1(\R^2;\R^2)$ which decreases along the trajectories of \eqref{eq:problem} according to
\begin{align*}
\frac{d}{dt}\E(f,g) = -\int_{\R^2} \left[ f \left| \nabla\Delta(Af+Bg) - \nabla(af+bg) \right|^2 + B g \left| \nabla\Delta(f+g) - b \nabla(f+g) \right|^2 \right]\, dx\,.
\end{align*}

It turns out that there is  an underlying structure in \eqref{eq:problem} which allows us to view it as a gradient flow for the energy $\E$ defined in \eqref{E} with respect to the 
$2-$Wasserstein distance in $\mathcal{P}_2(\R^2)\times\mathcal{P}_2(\R^2)$. Recall that  $\mathcal{P}_2(\R^2)$ is the set of Borel probability measures on $\R^2$ with 
finite second moment and that, given two Borel probability measures $\mu$ and $\nu$ in $\mathcal{P}_2(\R^2),$ the $2-$Wasserstein distance $W_2(\mu,\nu)$ on $\mathcal{P}_2(\R^2)$ is defined by
$$
W_2^2(\mu,\nu):=\inf_{\pi\in\Pi(\mu,\nu)}\int_{\R^4} |x-y|^2d\pi(x,y)\,,
$$
where $\Pi(\mu,\nu)$  is the set of all probability measures $\pi\in \mathcal{P}(\R^4) $ which have marginals $\mu$ and $\nu,$ that is, $\pi[U\times \R^2]=\mu[U] $ and $\pi[\R^2\times U]=\nu[U]$ for all measurable subsets $U$ of $\R^2.$  

This gradient flow structure was actually uncovered in \cite{LMxx} in the simplified case where the capillary forces are neglected ($A=B=0$) and shown to provide a convenient setting to construct weak solutions to \eqref{eq:problem}. This is thus the approach we shall use in this paper to prove the existence of nonnegative weak solutions to \eqref{eq:problem}, showing additionally the convergence of the variational approximation. It is worth noticing at this point that, if $g=0$, the system \eqref{eq:problem} reduces to the single equation
$$
\p_t f=\mathrm{div}\left[f \left( -A \nabla\Delta f +a\nabla f \right)\right],
$$
which also has a gradient flow structure with respect to the $2-$Wasserstein distance as already shown 
in \cite{Ot98,Ot01} for $A=0$ and \cite{MMS09} for $a=0$. 
Let us recall that, since the 
pioneering works \cite{JKO98, Ot98, Ot01}, several parabolic equations have been interpreted as gradient flows for Wasserstein distances,
 including the Fokker-Planck equation \cite{JKO98}, the porous medium equation \cite{Ot01}, second order nonlocal and/or degenerate parabolic equations \cite{Ag05,AS08,BCC08}, 
some kinetic equations \cite{CG04,CMV06}, and fourth order degenerate parabolic equations \cite{LMS12,MMS09}. 
Besides \eqref{eq:problem} there does not seem to be  many systems of parabolic partial differential equations endowed with a similar structure with the exception 
of the parabolic-parabolic chemotaxis Keller-Segel system which has a mixed Wasserstein-$L_2$ gradient flow structure \cite{BLxx,CLxx}.

\medskip

Before stating the main result of this paper, we introduce some notations: let $\cK$ be the convex subset of $\mathcal{P}_2(\R^2)$ defined by
\begin{equation*}
\cK := \left\{ h\in L_1(\R^2, (1+|x|^2) dx) \cap H^1(\R^2)\,:\, h\ge 0 \text{ a.e. and } \|h\|_1=1\right\}\,,
\end{equation*}
and set $\cK_2:=\cK\times\cK.$ We next recall that the functional $H$ defined by
\begin{equation}\label{eq:H}
H(h):=\int_{\R^2} h\ln (h)\, dx 
\end{equation}
is well-defined for $h\in\cK,$ see Lemma~\ref{L:A2}.

The main result of this paper is the following:

\begin{thm}\label{MT}
Assume that \eqref{constants} is satisfied and let $(f_0,g_0)\in\cK_2.$
Given $\tau\in(0,1),$ we define  $(f_\tau^0,g_\tau^0):=(f_0,g_0)$ and consider for each $n\in\N$ the minimization problem
\begin{equation}\label{tilt}
\inf_{(u,v)\in\cK_2} \cF_\tau^n(u,v)\,,
\end{equation}
where 
\begin{equation*}
\cF_\tau^n(u,v) := \frac{1}{2\tau}\ \left[ W_2^2(u,f_\tau^n)+B\ W_2^2(v,g_\tau^n) \right] + \E(u,v)\,, \quad (u,v)\in\cK_2.
\end{equation*}
For each $n\in\N$, the minimization problem \eqref{tilt} has a solution $(f_\tau^{n+1},g_\tau^{n+1})$, which is unique if  $a\geq b$. Defining the interpolation functions $(f_\tau,g_\tau)$ by $(f_\tau,g_\tau)(t) := (f_\tau^{n},g_\tau^{n})$ for $t\in[n\tau,(n+1)\tau)$ and $n\in\N$,
there exist a sequence $\tau_k\searrow 0$ and functions $(f,g):[0,\infty)\to\cK_2$ such that
\begin{equation}\label{dem1}
\text{$(f_{\tau_k},g_{\tau_k})(t)\to (f,g)(t) $ \quad in \quad$L_2(\R^2;\R^2)$}
\end{equation}
for all $t\geq0.$ 
Moreover,
\begin{itemize}
\item[$(i)$] $(f,g)\in L_\infty(0,t;H^1(\R^2;\R^2))\cap L_2(0,t; H^2(\R^2;\R^2)),$
\item[$(ii)$] $(f,g)\in C([0,\infty), L_2(\R^2;\R^2))$ and $(f,g)(0)=(f_0,g_0),$
\end{itemize}
and the pair $(f,g)$ is a weak solution of \eqref{eq:problem} in the sense that
\begin{equation}\label{dem2}
\left\{
\begin{array}{lll}
&\displaystyle \int_{\R^2}(f(t)-f_0)\xi\, dx+ \int_0^t\int_{\R^2} \left[ \Delta(A f+B g)\ \mathrm{div}(f\nabla\xi) + f \nabla(a f+ b g)\cdot \nabla\xi \right] dx\, ds=0\\[2ex]
& \displaystyle \int_{\R^2}(g(t)-g_0)\xi\, dx+ \int_0^t\int_{\R^2} \left[ \Delta( f+ g)\ \mathrm{div}(g\nabla\xi) + c g \nabla(f+g)\cdot \nabla\xi \right] dx\,ds=0
\end{array}
\right.
\end{equation}
for all $t>0$ and $\xi\in\mathcal{C}_0^\infty(\R^2).$
Finally, $(f,g)$ satisfies the following estimates:
\begin{align}\label{dem3}
& H(f(t)) + B H(g(t)) + \int_{0}^t D_H(f(s),g(s))\, ds  \leq H(f_0) + B H(g_0)\,,\\
\label{dem4}
& \mathcal{E}(f(t),g(t)) + \frac{1}{2}\ \int_{0}^t \left( \|w_f(s)\|_2^2 + B \|w_g(s)\|_2^2 \right)\, ds\leq \mathcal{E}(f_0,g_0)
\end{align}
for all $t>0,$ where
\begin{equation*}
D_H(f,g) := (A-B) \|\Delta f\|^2_2 + B \|\Delta(f+g)\|^2_2 + (a-b) \|\nabla f\|^2_2 + b \|\nabla(f+g)\|^2_2\,,
\end{equation*}
and $w_f$ and $w_g$ are two vector fields in $L_2((0,\infty)\times\R^2;\R^2)$ defined as follows: introducing the vector fields $(j_f,j_g)$ defined by 
\begin{equation}
\begin{array}{cl}
j_f := & - \nabla(f \Delta(Af+Bg)) + \Delta(A f+B g) \nabla f + f \nabla(a 
 f + b g)  \\
 & \\
 j_g := & - \nabla(g \Delta(f+g)) + \Delta(f+g) \nabla g + c g \nabla(
f+g)
\end{array} \quad\text{in $\mathcal{D}'((0,\infty)\times\R^2;\R^2),$} \label{hop}
\end{equation}
 they actually both belong to $L_2(0,\infty;L_{4/3}(\R^2;\R^2))$ and 
\begin{equation}
j_f = \sqrt{f}\ w_ f \;\;\text{ and }\;\; j_g = \sqrt{g}\ w_g \;\;\text{ a.e. in }\;\; (0,\infty)\times\R^2\,. \label{hip}
\end{equation}
\end{thm}

\medskip
Notice that Theorem~\ref{MT} provides not only the existence of a weak solution to \eqref{eq:problem} but also 
the convergence of subsequences of solutions to the variational scheme \eqref{tilt} to a solution to \eqref{eq:problem}. 
The proof of Theorem~\ref{MT} thus relies strongly on the study of the minimization problem \eqref{tilt} which is performed in Section~\ref{S:1}. 
The availability of the second Liapunov functional $H(f)+B H(g)$ (which is really a Liapunov functional only if $a\ge b$) plays an important role here as it allows us to show by an argument from \cite{MMS09} 
that the minimizers of \eqref{tilt} are actually in $H^2(\R^2)$.  
This additional regularity is actually at the heart of the identification of the Euler-Lagrange equation satisfied by the minimizers, see Section~\ref{Sec:2}. 
The convergence of the scheme and the existence of weak solutions to \eqref{eq:problem} are established in the last section, some care being needed to handle the fourth order terms. 
 Indeed, as indicated in \eqref{hop} and \eqref{hip}, due to the degenerate character of the equations, we have to  deal with  expressions that correspond to the fourth order terms in \eqref{Prob} and are not functions.

\begin{rem}
\begin{enumerate}
\item In contrast to the dissipation of the energy $\E$, the dissipation $D_H(f,g)$ of the functional $H(f)+B H(g)$ need not be nonnegative and this occurs when $a<b$, see \eqref{dem3}. 
It is yet unclear what information on the dynamics may be retrieved from \eqref{dem3} in that case.
\item For simplicity, we have restricted the analysis to initial data $(f_0,g_0)$ satisfying $\|f_0\|_1=\|g_0\|_1=1$. 
A similar result holds true for arbitrary nonnegative and integrable initial data $(f_0,g_0)\in H^1(\R^2;\R^2)$ with finite second moment and may be proved analogously to Theorem~\ref{MT}. Introducing   $(F,G):=(f/\|f_0\|_1,g/\|g_0\|_1)$, this new unknown  solves a system with the same structure as \eqref{eq:problem}, but with different parameters (depending on $\|f_0\|_1$ and $\|g_0\|_1$).
\item Theorem~\ref{MT} is also valid for the one dimensional version of \eqref{eq:problem}.
\end{enumerate}
\end{rem}

Throughout the paper we use the following notation: $x=(x_1,x_2)$ denotes a generic point of $\R^2$ and the partial derivative with respect to $x_i$ is denoted by $\partial_i$, $i=1,2$. 
For a sufficiently smooth function $h$, $\nabla h :=(\partial_1 h, \partial_2 h)$ denotes its gradient and $D^2 h := \left( \partial_i \partial_j h \right)_{1\le i, j\le 2}$ its Hessian matrix. 
Finally, $D\xi$ denotes the gradient of a vector field $\xi=(\xi_1,\xi_2)\in C^1(\R^2;\R^2)$ and is given by $D\xi := \left( \partial_i \xi_j \right)_{1\le i,j\le 2}$, 
 the divergence of $\xi$ being the trace of $D\xi$, that is, $\mathrm{div}(\xi)=\partial_1 \xi_1 + \partial_2 \xi_2$.

\section{The minimization problem}\label{S:1}

In this section, we show that, given $(f_0,g_0)\in\cK_2$ and $\tau>0$, the minimization problem 
\begin{equation}\label{Zmin}
\inf_{(u,v)\in\cK_2}\cF_\tau(u,v)
\end{equation}
with the  functional $\cF_\tau$ defined by
\begin{equation}\label{Zfunctional}
\cF_\tau(u,v) := \frac{1}{2\tau}\ \left(W_2^2(u,f_0)+B\ W_2^2(v,g_0) \right) + \E(u,v)\,, \quad (u,v)\in\cK_2\,,
\end{equation}
has at least one solution and we study the regularity of the minimizers. 
Since $a-b$ might be negative, which is the case when the more dense fluid lies on top of the less dense one, cf. \cite{EMxx}, the first step is to prove that $\E$ is bounded from below.

\begin{lemma}\label{L:0} The energy functional $\E$ satisfies
\begin{equation}\label{eq:3}
\E(f,g)\geq -C_1 +\frac{B}{2}\ \|\nabla (f+ g)\|^2_2 + \frac{A-B}{4}\ \|\nabla f\|^2_2\qquad\text{for all $(f,g)\in\cK_2,$}
\end{equation}
where $C_1$ is a positive constant depending only $A$, $B$, $a$, $b$, and $c$.
\end{lemma}

\begin{proof}
According to the Gagliardo--Nirenberg--Sobolev inequality there is a positive constant $C_2>0$ such that 
\begin{equation}\label{eq:1}
 \|h\|_{2}\leq C_2\ \|h\|_1^{1/2}\ \|\nabla h\|_{2}^{1/2} \qquad \text{for all $h\in H^1(\R^2).$}
\end{equation}
By the definition of $\cK$, this inequality yields
\begin{equation}\label{eq:2}
\|h\|_2^2\leq C_2^2\ \|\nabla h\|_2\qquad \text{for all $h\in \cK.$}
\end{equation}
Owing to \eqref{eq:2} and Young's inequality, we find, for $(f,g)\in\cK_2$,
\begin{align*}
2\E(f,g)=& (A-B)\ \|\nabla f\|_2^2 + B\ \|\nabla (f+ g)\|_2^2 + (a-b)\ \|f\|_2^2 + b\ \|f+g\|_2^2 \nonumber\\
\geq&B\ \|\nabla (f+ g)\|^2_2 + (A-B)\ \|\nabla f\|_2^2 - C_2^2 (b-a)_+\ \|\nabla f\|_2 \nonumber\\
\geq&B\ \|\nabla (f+ g)\|^2_2 + \frac{A-B}{2}\ \|\nabla f\|_2^2 \nonumber\\
& + \frac{A-B}{2}\ \left( \|\nabla f\|_2 - \frac{C_2^2 (b-a)_+}{A-B}  \right)^2 - \frac{C_2^4 (b-a)_+^2}{2(A-B)} \,,
\end{align*}
whence \eqref{eq:3}. 
\end{proof}

We now show the existence of a minimizer to \eqref{Zmin}.

\begin{lemma}\label{Unimin}
 Given $(f_0,g_0)\in\cK_2$ and $\tau>0,$ there exists a  minimizer $(f,g)\in\cK_2$ of \eqref{Zmin} which is unique if $a\ge b$.
Moreover, $f$ and $g$ both belong to $H^2(\R^2)$ and 
\begin{equation}\label{Heat3}
\begin{aligned}
&(A-B)\ \|\Delta f\|_2^2 + B\ \|\Delta (f+g)\|_2^2 + (a-b)\ \|\nabla f\|_2^2+b\|\nabla (f+g)\|_2^2\\
&\leq\frac{1}{\tau}\ \left[ H(f_0)-H(f) +B\left( H(g_0)-H(g) \right) \right],
\end{aligned}
\end{equation}
the functional $H$ being defined in \eqref{eq:H}.
\end{lemma}

\begin{proof}
Pick a minimizing sequence $(u_k,v_k)_{k\ge 1}$ of $\cF_\tau$ in $\cK_2.$
Invoking \eqref{eq:3} and \eqref{eq:2} we see that
\begin{eqnarray}
\| u_k\|_{H^1} + \|v_k\|_{H^1} & \le & C\,, \quad k\ge 1\,, \label{eq:4} \\
W_2(u_k,f_0) + W_2(v_k,g_0) & \le & C\,, \quad k\ge 1\,, \label{eq:5} 
\end{eqnarray}
and the estimate \eqref{eq:5} implies that
\begin{equation}\label{eq:6}
\int_{\R^2}  (u_k+v_k)(x) (1+|x|^2)\, dx \le C\,, \qquad k\ge 1\,,
\end{equation}
by classical properties of the Wasserstein distance $W_2$, see, e.g., \cite[Eq.~(17)]{JKO98}. The compactness of the embedding of $H^1(\R^2)\cap L_1(\R^2;|x|^2\, dx)$ in $L_2(\R^2)$ 
(which is recalled in Lemma~\ref{le:comp} below) ensures that there are non-negative functions $f$ and $g$ in $H^1(\R^2)\cap L_1(\R^2, (1+|x|^2)\, dx) $ and a subsequence of $(u_k,v_k)_{k\ge 1}$ (not relabeled) having the property that
\begin{equation}\label{conv}
\begin{array}{cl}
(u_k,v_k) \to (f,g) & \text{ in $L_2(\R^2;\R^2)$ and a.e. in $\R^2$,}\\
& \\
(u_k,v_k) \rightharpoonup (f,g) & \text{ in $H^1(\R^2;\R^2).$}
\end{array}
\end{equation}
Furthermore, it readily follows  from \eqref{eq:4}, \eqref{eq:6}, and the Dunford-Pettis theorem that $(u_k)_{k\ge 1}$ and $(v_k)_{k\ge 1}$ are weakly sequentially compact in $L_1(\R^2)$. 
This property, together with \eqref{conv} and the Vitali theorem imply the strong convergence of $(u_k)_{k\ge 1}$ and $(v_k)_{k\ge 1}$ in $L_1(\R^2)$, namely,
\begin{equation}\label{conv2}
(u_k,v_k) \to (f,g) \quad \text{ in $L_1(\R^2;\R^2)$}.
\end{equation}
Since $(u_k,v_k)_{k\ge 1}$ belongs to $\cK_2$ for each $k\geq 1$, we conclude from \eqref{eq:6}, \eqref{conv}, and \eqref{conv2}  that $(f,g)\in\cK_2.$
That $(f,g)$ is a minimizer of $\cF_\tau$ in $\cK_2$ follows from \eqref{conv}, \eqref{conv2}, the weak lower semicontinuity of $\E$, and the fact that
 the $2$-Wasserstein metric $W_2$ is lower semicontinuous with respect to the narrow convergence of  probability measures in each of its arguments, the latter convergence
 being guaranteed by \eqref{conv2}. This completes the proof of the existence of a minimizer to $\cF_\tau$ in $\cK_2$.
 
Next, when $a\ge b,$ the uniqueness of the minimizer follows from the convexity of $\cK_2$ and $W_2^2$ and the strict convexity of  $\E$. 

We finish off the proof by showing that any minimizer $(f, g)$ of $\cF_\tau$ in $\cK_2$ actually belongs to $H^2(\R^2;\R^2),$ the proof relying on a technique 
developed in \cite{MMS09}. 
More precisely, denoting the heat semigroup by $(G_s)_{s\ge 0}$ which is defined by
$$
(G_s h)(x) := \frac{1}{4\pi s}\ \int_{\R^2} \exp{\left( - \frac{|x-y|^2}{4s} \right)}\ h(y)\, dy\,, \quad (s,x)\in [0,\infty)\times\R^2\,,
$$ 
for $h\in L_1(\R^2)$, classical properties of the heat semigroup ensure that $(G_s f, G_s g)\in\cK_2$ for all $s\ge 0$ since $(f,g)\in\cK_2.$ 
Consequently,  $\cF_\tau(f, g) \leq \cF_\tau(G_s f, G_s g)$ and we deduce that
\begin{equation}\label{Heat-1}
\E(f,g) - \E(G_s f,G_s g) \leq \frac{1}{2\tau}\left[\left(W_2^2(G_s f,f_0)-W_2^2(f,f_0)\right) + B\left(W_2^2(G_s g,g_0)-W_2^2(g,g_0)\right)\right]
\end{equation}
for all $s\ge 0.$ On the one hand, an explicit computation gives for $s>0$
\begin{align*}
\frac{d}{ds} \E(G_s f,G_s g)=&\int_{\R^2} \left[ (A-B)\ \nabla G_s f\ \p_s(\nabla G_s f) + B\ \nabla G_s (f+g)\ \p_s(\nabla G_s (f+g))\right.\\
&\left.\hspace{0.5cm} + (a-b)\ G_s f\ \p_s G_s f + b\ G_s (f+g)\ \p_s G_s(f+g) \right]\, dx\\
=&-(A-B)\ \|\Delta G_s f\|_2^2 - B\ \|\Delta G_s(f+g)\|_2^2\\
& - (a-b)\ \|\nabla G_s f\|_2^2 - b\ \|\nabla G_s(f+g)\|_2^2\,,
\end{align*}
which yields, after integration with respect to time,
\begin{align*}
&\frac{1}{s} \int_0^s \left[ (A-B)\ \|\Delta G_\sigma f\|_2^2 + B\ \|\Delta G_\sigma (f+g)\|_2^2 + (a-b)\ \|\nabla G_\sigma  f\|_2^2 + b\ \|\nabla G_\sigma (f+g)\|_2^2 \right] \, d\sigma \\
&=\frac{\E(f,g)-\E(G_s f,G_s g)}{s}\,.
\end{align*}
Since $\sigma\mapsto \|\Delta G_\sigma h\|_2$ and $\sigma\mapsto \|\nabla G_\sigma h\|_2$ are non-increasing functions for all $h\in L_1(\R^2)$, we end up with
\begin{equation}\label{Heat0}
\begin{aligned}
&(A-B)\ \|\Delta G_s f\|_2^2 + B\ \|\Delta G_s (f+g)\|_2^2 + (a-b)_+\ \|\nabla G_s f\|_2^2 + b\ \|\nabla G_s(f+g)\|_2^2 \\
&\leq  \frac{\E(f,g)-\E(G_s f,G_s g)}{s} + (b-a)_+\ \|\nabla f\|_2^2
\end{aligned}
\end{equation}
for all $s>0.$ 
On the other hand, since the heat equation is the gradient flow of the entropy functional $H$ defined by \eqref{eq:H} for the $2$-Wasserstein distance $W_2$, 
see, e.g., \cite{AGS08,JKO98,Ot01,Vi03}, it follows from \cite[Theorem~11.1.4]{AGS08} that, for all  $(h,\tilde{h})\in \cK_2,$
\begin{equation}\label{Heat}
\frac{1}{2}\frac{d}{ds}W_2^2(G_s h,\tilde{h}) + H(G_s h)\leq H(\tilde{h}) \qquad \text{for a.e. $s\geq0$}.
\end{equation}
With the choices $(h,\tilde{h})=(f,f_0)$ and $(h,\tilde{h})=(g,g_0)$ in \eqref{Heat}, we obtain 
$$
\frac{1}{2}\frac{d}{ds}\left[ W_2^2(G_s f,f_0)+B\ W_2^2(G_s g,g_0) \right] \leq H(f_0)-H(G_s f) +B\ \left( H(g_0)-H(G_s g) \right)
$$
for a.e. $s\ge 0.$ 
Integrating the above inequality with respect to time and using the time monotonicity of $s\mapsto H(G_s f)$ and $s\mapsto H(G_s g)$ lead us to
\begin{equation}\label{Heat1}
\begin{aligned}
&  \frac{1}{2s} \left[ W_2^2(G_s f,f_0) - W_2^2(f,f_0) +B\ \left( W_2^2(G_s  g,g_0) - W_2^2(g,g_0) \right) \right] \\
&  \leq \left[ H(f_0)-H(G_s f) + B\ \left( H(g_0)-H(G_s g) \right) \right]\,. 
\end{aligned}
\end{equation}
Combining \eqref{Heat-1}, \eqref{Heat0}, and \eqref{Heat1}, we find
\begin{equation}\label{Heat2}
\begin{aligned}
&(A-B)\ \|\Delta G_s f\|_2^2 + B\ \|\Delta G_s (f+g)\|_2^2 + (a-b)_+\ \|\nabla G_s f\|_2^2 + b\ \|\nabla G_s (f+g)\|_2^2\\
&\leq \frac{1}{\tau}\ \left[ H(f_0)-H(G_s f) + B\ \left( H(g_0) - H(G_s g) \right) \right] + (b-a)_+\ \|\nabla f\|_2^2
\end{aligned}
\end{equation}
for $s>0.$ Since $(f,g)\in \cK_2$, 
classical properties of the heat semigroup and the functional $H$ entail that $(H(G_s f),H(G_s g))$ converges towards $(H(f),H(g))$ as $s\to 0$. 
This convergence, \eqref{Heat2}, and the positivity of $A-B$ and $B$ readily imply that $(\Delta G_s f)_{s>0}$ and $(\Delta G_s (f+g))_{s>0}$ 
are bounded in $L_2(\R^2)$. Since they converge towards $\Delta f$ and $\Delta (f+g)$, respectively, in the sense of distributions as $s\to 0$,
 we deduce that both $\Delta f$ and $\Delta (f+g)$ belong to $L_2(\R^2)$ and $(\Delta G_s f,\Delta G_s (f+g))_{s>0}$ converges weakly to $(\Delta f,\Delta (f+g))$ in $L_2(\R^2;\R^2)$ as $s\to 0$. 
As a consequence, recalling that $(f,g)\in\cK_2$,  we conclude that $f$ and $f+g$ belong to $H^2(\R^2),$ and so does $g$. It remains to pass to the limit $s\to 0$ in \eqref{Heat2} to obtain \eqref{Heat3}.
\end{proof}

\section{The Euler-Lagrange equations}\label{Sec:2}

We now identify the Euler-Lagrange equation satisfied by the minimizers of the functional $\cF_\tau$, defined in \eqref{Zfunctional}, in $\cK_2$. 

\begin{lemma}\label{le:sup1}
Consider $(f_0,g_0)\in\cK_2$ and $\tau>0$. If $(f,g)$  is a minimizer of $\cF_\tau$ in $\cK_2$, it satisfies
\begin{equation}\label{A1}
\begin{aligned}
&\left| \int_{\R^2}(f-f_0)\xi\, dx + \tau\ \int_{\R^2} \left[\Delta(A f+B g)\ \mathrm{div}(f\nabla\xi) + f\ \nabla(a f+ b g)\cdot \nabla\xi\right] dx\right| \\
& \hspace{8cm} \leq \ \frac{\|D^2\xi\|_{\infty}}{2}\ W_2^2(f,f_0)
\end{aligned}
\end{equation}
and
\begin{equation}\label{A2}
\begin{aligned}
& \left| \int_{\R^2}(g-g_0)\xi\, dx + \tau\ \int_{\R^2} \left[ \Delta(f+ g)\ \mathrm{div}(g\nabla\xi) + c g\ \nabla(f+g)\cdot \nabla\xi\right] dx \right|\\ & \hspace{9cm} \leq \frac{\|D^2\xi\|_{\infty}}{2}\ W_2^2(g,g_0)
\end{aligned}
\end{equation}
for $\xi\in C_0^\infty(\R^2)$, where $D^2\xi$ denotes the Hessian matrix of $\xi.$
\end{lemma}

\begin{proof}
By Brenier's theorem \cite[Theorem~2.12]{Vi03}, there are two measurable functions $T:\R^2\to\R^2$ and $S:\R^2\to\R^2$ such that $(f,g)=(T\#f_0,S\#g_0)$ and
\begin{equation}
W_2^2(f,f_0) = \int_{\R^2} |x-T(x)|^2\ f_0(x) dx \;\;\mbox{ and }\;\; W_2^2(g,g_0) = \int_{\R^2} |x-S(x)|^2\ g_0(x) dx\,. \label{new1}
\end{equation}
We pick now two test functions $\eta=(\eta_1,\eta_2)$ and $\xi=(\xi_1,\xi_2)$ in $C_0^\infty(\R^2;\R^2) $ and define 
\begin{equation}\label{eq:app2}
\begin{aligned}
(T_\e,S_\e) & := (\id+\e\xi,\id+\e\eta)\,, \\
(f_\e,g_\e) & := \big( (T_\e\circ T) \# f_0, (S_\e\circ S) \# g_0 \big) = \big( T_\e \# f, S_\e \# g \big)\,,
\end{aligned}
\end{equation}
for each $\e\in[0,1]$, where $\id$ is the identity function on $\R^2.$ Clearly, there is $\e_0\in (0,1)$ small enough
 (depending on both $\xi$ and $\eta$) such that, for $\e\in[0,\e_0],$ $T_\e$ and $S_\e$ are $C^\infty-$diffeomorphisms from $\R^2$ onto $\R^2$ with positive Jacobi determinants
\begin{equation}
\begin{aligned}
{\J_\e^\xi} & := \mathrm{det}(D T_\e) = 1 + \e\ \mathrm{div}(\xi) + \e^2\ \mathrm{det}(D\xi)\,, \\ {\J_\e^\eta} & := \mathrm{det}(D S_\e) = 1 + \e\ \mathrm{div}(\eta) + \e^2\ \mathrm{det}(D\eta)\,.
\end{aligned}
\label{eq:det}
\end{equation}
By \eqref{eq:app2}, we have the identities
\begin{equation}\label{eq:app} 
f_\e = \frac{f\circ T_\e^{-1}}{{\J_\e^\xi}\circ T_\e^{-1}} \quad\text{ and }\quad 
g_\e = \frac{g\circ S_\e^{-1}}{{\J_\e^\eta} \circ S_\e^{-1}},  \qquad \e\in (0,\e_0],
\end{equation}
from which we deduce that $f_\e$ and $g_\e$ both belong to $H^2(\R^2)$ and satisfy $\|f_\e\|_1=\|f\|_1=\|g_\e\|_1=\|g\|_1=1$.
 Additionally, we compute
\begin{equation}\label{2-1}
\|f_\e\|_2^2=\int_{\R^2} \frac{f^2(x)}{{\J_\e^\xi} (x)}\, dx 
\qquad\text{and}\qquad \|g_\e\|_2^2 = \int_{\R^2}\frac{g^2(x)}{{\J_\e^\eta}(x)}\, dx.
\end{equation}

For further use, we now study the behaviour of $(f_\e)_\e$ and $(g_\e)_\e$ as $\e\to 0$. To begin with, we notice that, given a test function $\vartheta\in C_0^\infty(\R^2)$, it follows from \eqref{eq:det} and \eqref{eq:app} that 
\begin{equation}
\lim_{\e\to 0} \int_{\R^2} f_\e\ \vartheta\ dx = \lim_{\e\to 0} \int_{\R^2} f\ \left( \vartheta\circ T_\e \right)\, dx = \int_{\R^2} f\ \vartheta\, dx \label{tintin}
\end{equation}
and
\begin{equation}
\lim_{\e\to 0} \int_{\R^2} \frac{f_\e-f}{\e}\ \vartheta\ dx = \lim_{\e\to 0} \int_{\R^2} f\ \frac{\vartheta\circ T_\e-\vartheta}{\e}\ dx = \int_{\R^2} f\ \nabla\vartheta\cdot\xi \ dx\,, \label{milou}
\end{equation}
while \eqref{eq:det}, \eqref{2-1}, and the Lebesgue dominated convergence theorem entail that
\begin{equation}
\lim_{\e\to 0} \|f_\e\|_2^2 = \|f\|_2^2\,. \label{haddock}
\end{equation}
Thanks to \eqref{tintin} and \eqref{haddock} on the one hand and to \eqref{milou} and Lemma~\ref{L:2} on the other hand, we conclude that
\begin{equation}
f_\e \rightarrow f \;\mbox{ in }\; L_2(\R^2) \;\;\mbox{ and }\;\; \frac{f_\e-f}{\e} \rightharpoonup - \mathrm{div}(f \xi) \;\mbox{ in }\; L_2(\R^2)\,. \label{eq:cvfe}
\end{equation}
Similarly, we get
\begin{equation}
g_\e \rightarrow g \;\mbox{ in }\; L_2(\R^2) \;\;\mbox{ and }\;\; \frac{g_\e-g}{\e} \rightharpoonup - \mathrm{div}(g \eta) \;\mbox{ in }\; L_2(\R^2)\,. \label{eq:cvge}
\end{equation}

We next turn to the convergence properties of $(\nabla f_\e)_\e$ and $(\nabla g_\e)_\e$. Differentiating \eqref{eq:app} and using \eqref{eq:det}, we find
\begin{equation}
\nabla f_\e\circ T_\e = \frac{1}{(\J_\e^\xi)^2}\ \nabla f + \e\ V_\e(f,\xi) - \e^2\ \frac{f}{(\J_\e^\xi)^3}\ R_\e(\xi)\,, \label{d3}
\end{equation}
where $V_\e(f,\xi) := \left( V_{1,\e}(f,\xi),V_{2,\e}(f,\xi) \right)$, $R_\e(\xi) := \left( R_{1,\e}(\xi),R_{2,\e}(\xi) \right)$, 
\begin{align*}
V_{1,\e}(f,\xi) & := \frac{\nabla f \cdot \left( \partial_2 \xi_2 , -\partial_1 \xi_2 \right)}{(\J_\e^\xi)^2} - \frac{f\ \partial_1\mathrm{div}(\xi)}{(\J_\e^\xi)^3}\,, \\
V_{2,\e}(f,\xi) & := \frac{\nabla f \cdot \left( -\partial_2 \xi_1 , \partial_1 \xi_1 \right)}{(\J_\e^\xi)^2} - \frac{f\ \partial_2\mathrm{div}(\xi)}{(\J_\e^\xi)^3}\,, 
\end{align*}
and
\begin{align*}
R_{1,\e}(\xi) & := \partial_1\mathrm{det}(D\xi) + \partial_1\mathrm{div}(\xi)\ \partial_2 \xi_2 - \partial_2\mathrm{div}(\xi)\ \partial_1 \xi_2 + \e \left( \partial_1\mathrm{det}(D\xi)\ \partial_2 \xi_2 - \partial_2\mathrm{det}(D\xi)\ \partial_1 \xi_2 \right)\,, \\
R_{2,\e}(\xi) & := \partial_2\mathrm{det}(D\xi) + \partial_2\mathrm{div}(\xi)\ \partial_1 \xi_1 - \partial_1\mathrm{div}(\xi)\ \partial_2 \xi_1 + \e \left( \partial_2\mathrm{det}(D\xi)\ \partial_1 \xi_1 - \partial_1\mathrm{det}(D\xi)\ \partial_2 \xi_1 \right)\,.
\end{align*}
Let us now draw several consequences of \eqref{d3}: first, we have 
$$
\|\nabla f_\e\|_2^2 = \int_{\R^2} \left| \nabla f_\e \circ T_\e \right|^2\ \J_\e^\xi\ dx \,,
$$
and, since $f\in H^1(\R^2)$ and $\J_\e^\xi\longrightarrow 1$ in $L_\infty(\R^2)$ by \eqref{eq:det}, it readily follows from \eqref{d3} and the previous identity that
\begin{equation}
\lim_{\e\to 0} \| \nabla f_\e \|_2^2 = \|\nabla f\|_2^2\,. \label{d4}
\end{equation}
In particular, $(\nabla f_\e)_\e$ is bounded in $L_2(\R^2;\R^2)$ and, since $f_\e\to f$ in $L_2(\R^2)$ by \eqref{eq:cvfe}, we conclude that $(\nabla f_\e)_\e$ converges weakly towards $\nabla f$ in $L_2(\R^2;\R^2)$. This convergence and \eqref{d4} then guarantee that
\begin{equation}
\nabla f_\e \to \nabla f \;\;\mbox{ in }\;\; L_2(\R^2;\R^2)\,. \label{d5}
\end{equation} 
Next, owing to \eqref{d3}, we have
\begin{align*}
\frac{1}{\e}\ \nabla(f_\e - f) = & \left[ \frac{1 - \J_\e^\xi \circ T_\e^{-1}}{\e} \right]\ \frac{\nabla f\circ T_\e^{-1}}{\left( \J_\e^\xi\circ T_\e^{-1} \right)^2} + \frac{1}{\e}\ \left[ \frac{\nabla f\circ T_\e^{-1}}{\J_\e^\xi\circ T_\e^{-1}} - \nabla f \right] \\ 
& + V_\e(f,\xi)\circ T_\e^{-1} - \e\ \frac{f\circ T_\e^{-1}}{\left( \J_\e^\xi\circ T_\e^{-1}\right)^3}\ R_\e(\xi)\circ T_\e^{-1}\,.
\end{align*}
The properties of $\J_\e^\xi$, $T_\e$, $f$, and the definitions of $V_\e(f,\xi)$ and $R_\e(\xi)$ readily ensure that the first, third and fourth terms on the right-hand side of the above identity are
 bounded in $L_2(\R^2;\R^2)$ while the boundedness in $L_2(\R^2;\R^2)$ of the second one follows from Lemma~\ref{L:2} since $f\in H^2(\R^2)$ by Lemma~\ref{Unimin}. 
Consequently, $(\nabla(f_\e-f)/\e)_\e$ is bounded in $L_2(\R^2;\R^2)$ and, recalling that $((f_\e-f)/\e)_\e$ converges weakly to $-\mathrm{div}(f\xi)$ in $L_2(\R^2;\R^2)$ by \eqref{eq:cvfe}, we conclude that
\begin{equation}
\frac{\nabla(f_\e-f)}{\e} \rightharpoonup - \nabla\mathrm{div}(f\xi) \;\;\mbox{ in }\;\; L_2(\R^2;\R^2)\,. \label{d6}
\end{equation}
Similar computations give
\begin{equation}
\nabla g_\e \to \nabla g \;\mbox{ in }\; L_2(\R^2;\R^2) \;\;\mbox{ and }\;\; \frac{\nabla(g_\e-g)}{\e} \rightharpoonup - \nabla\mathrm{div}(g\eta) \;\mbox{ in }\; L_2(\R^2;\R^2)\,. \label{d7}
\end{equation}

After this preparation, we can start the proof of \eqref{A1} and \eqref{A2}. Recalling that $(f_\e,g_\e)\in\cK_2$ for all $\e\in(0,\e_0]$, the minimizing property of  $(f,g)$ entails that $\cF_\tau(f,g) \le \cF_\tau(f_\e,g_\e)$, that is,
\begin{equation}\label{inequ}
0 \le \frac{1}{2\tau}\ \left[ W_2^2(f_\e,f_0) - W_2^2(f,f_0) + B\left( W_2^2(g_\e,g_0) - W_2^2(g,g_0) \right) \right] + \E(f_\e,g_\e) - \E(f,g).
\end{equation}
Since $(T_\e\circ T)\#f_0=f_\e$ by \eqref{eq:app2}, we infer from \eqref{new1}  that
\begin{equation*}
W_2^2(f_\e,f_0) \leq  W_2^2(f,f_0) - 2 \e \int_\R (\id-T)\cdot (\xi\circ T)\ f_0\, dx + \e^2 \int_\R |\xi\circ T|^2 f_0\, dx.
\end{equation*}
A similar inequality being valid with $(g_\e,g_0,g,S_\e,S,\eta)$ instead of  $(f_\e,f_0,f,T_\e,T,\xi)$, we conclude that
\begin{equation}\label{EQ1}
\begin{aligned}
 &\limsup_{\e\to0} \left\{ \frac{1}{2\e} \left[ W_2^2(f_\e,f_0) - W_2^2(f,f_0) + B \left( W_2^2(g_\e,g_0) - W_2^2(g,g_0) \right) \right] \right\} \\
&\leq - \left[ \int_{\R^2} (\id-T)\cdot (\xi\circ T)\ f_0\, dx + B \int_{\R^2} (\id-S)\cdot (\eta\circ S)\ g_0\, dx\right].
\end{aligned}
\end{equation}

We next show that
\begin{equation}\label{EQ2}
\begin{aligned}
& \lim_{\e\to 0} \frac{1}{2\e} \left[ (a-b)\ \|f_\e\|_2^2 + b\ \|f_\e+g_\e\|_2^2 - (a-b)\ \|f\|_2^2 - b\ \|f+g\|_2^2 \right] \\
&=-\int_{\R^2} \left[ a\ \frac{f^2}{2}\ \mathrm{div}(\xi) + b\ \frac{g^2}{2}\ \mathrm{div}(\eta) + b\ f \ \mathrm{div}(g \eta ) + b\ g \ \mathrm{div}(f \xi ) \right]\, dx\,.
\end{aligned}
\end{equation}
Indeed, we write
$$
(a-b)\ \|f_\e\|_2^2 + b\ \|f_\e+g_\e\|_2^2 - (a-b)\ \|f\|_2^2 - b\ \|f+g\|_2^2=a\ I_1^\e + b\ I_2^\e + b\ I_3^\e,
$$
with
$$
I_1^\e := \|f_\e\|_2^2 - \|f\|_2^2\, , \quad I_2^\e := \|g_\e\|_2^2 - \|g\|_2^2\,,\quad  I_3^\e := 2 \int_{\R^2} \left( f_\e\ g_\e - f\ g \right)(x)\, dx.
$$
It readily follows from \eqref{eq:cvfe} that
\begin{equation}\label{EQ:21}
\lim_{\e\to 0} \frac{I_1^\e}{2\e} = \lim_{\e\to 0} \int_{\R^2} \frac{(f_\e+f)}{2}\ \frac{(f_\e-f)}{\e}\, dx = - \int_{\R^2} f\ \mathrm{div}(f\xi)\, dx = -\frac{1}{2} \int_{\R^2} f^2 \ \mathrm{div}(\xi)\, dx\,.
\end{equation}
Similarly, \eqref{eq:cvge} guarantees that
\begin{equation}\label{EQ:22}
\lim_{\e\to 0} \frac{I_2^\e}{2\e} = \lim_{\e\to 0} \int_{\R^2} \frac{(g_\e+g)}{2}\ \frac{(g_\e-g)}{\e}\, dx = -\frac{1}{2} \int_{\R^2} g^2\ \mathrm{div}(\eta)\, dx. 
\end{equation}
We next write $I_3^{\e}$ as
$$
I_3^\e := \int_{\R^2} \left[ (f_\e+f)\ (g_\e-g) + (g_\e+g)\ (f_\e-f) \right]\, dx\,,
$$
and deduce from \eqref{eq:cvfe} and \eqref{eq:cvge} that 
\begin{equation}\label{EQ:23}
\lim_{\e\to 0} \frac{I_3^\e}{2\e} = - \int_{\R^2} \left[ f\ \mathrm{div}(g\eta) + g\ \mathrm{div}(f\xi) \right]\, dx.
\end{equation}
Combining \eqref{EQ:21}, \eqref{EQ:22}, and \eqref{EQ:23} gives the claim \eqref{EQ2}.

Finally, we show that
\begin{equation}\label{EQ3}
\begin{aligned}
 &\lim_{\e\to 0} \frac{1}{2\e} \left[ (A-B)\ \|\nabla f_\e\|_2^2 + B\ \|\nabla (f_\e+g_\e) \|_2^2 - (A-B)\ \|\nabla f\|_2^2 - B\ \|\nabla (f+g)\|_2^2 \right] \\
&=\int_{\R^2} \left[ (A\Delta f+ B\Delta g)\ \mathrm{div}(f\xi) + B \Delta(f+g)\ \mathrm{div}(g\eta) \right]\, dx.
\end{aligned}
\end{equation}
To this end, we write
$$
(A-B)\ \|\nabla f_\e\|_2^2 + B\ \|\nabla(f_\e+g_\e)\|_2^2 - (A-B)\ \|\nabla f\|_2^2 - B\ \|\nabla(f+g)\|_2^2 = A\ L_1^\e + B\ L_2^\e + B\ L_3^\e,
$$
with
$$
L_1^\e := \|\nabla f_\e\|_2^2 - \|\nabla f\|_2^2\,, \quad L_2^\e := \|\nabla g_\e\|_2^2 - \|\nabla g\|_2^2\,, \quad L_3^\e := 2\ \int_{\R^2} \left( \nabla f_\e\cdot \nabla g_\e - \nabla f\cdot \nabla g \right)\, dx.
$$
Thanks to \eqref{d5} and \eqref{d6}, we have
\begin{equation}
\begin{aligned}
\lim_{\e\to 0} \frac{L_1^\e}{2\e} = &\lim_{\e\to 0} \int_{\R^2} \frac{\nabla(f+f_\e)}{2}\cdot \frac{\nabla(f_\e-f)}{\e}\, dx \\
= & - \int_{\R^2} \nabla f\cdot \nabla\mathrm{div}(f\xi) \, dx = \int_{\R^2}\Delta f\ \mathrm{div}(f\xi)\, dx\,,
\end{aligned}\label{d8}
\end{equation}
and similarly, by \eqref{d7},
\begin{equation}
 \lim_{\e\to 0} \frac{L_2^\e}{2\e} = \lim_{\e\to 0} \int_{\R^2} \frac{\nabla(g+g_\e)}{2}\cdot \frac{\nabla(g_\e-g)}{\e}\, dx = \int_{\R^2}\Delta g\ \mathrm{div}(g\eta)\, dx. \label{d9}
\end{equation}
Finally,
$$
L_3^\e = \int_{\R^2} \left[ \nabla(f_\e+f)\cdot \nabla(g_\e-g) + \nabla(f_\e-f)\cdot \nabla(g_\e+g) \right]\, dx\,,
$$
and we infer from \eqref{d5}, \eqref{d6}, and \eqref{d7} that 
\begin{equation}
\begin{aligned}
\lim_{\e\to 0} \frac{L_3^\e}{2\e} =& - \int_{\R^2} \left[ \nabla f\cdot \nabla\mathrm{div}(g\eta) + \nabla\mathrm{div}(f\xi)\cdot \nabla g \right]\, dx\\
=& \int_{\R^2} \left[ \Delta f\ \mathrm{div}(g\eta) + \Delta g\ \mathrm{div}(f\xi) \right]\, dx.
\end{aligned}\label{d10}
\end{equation}
Combining \eqref{d8}, \eqref{d9}, and \eqref{d10} gives the claim \eqref{EQ3}.

We now divide \eqref{inequ} by $\e$ and take the limsup as $\e\to 0$, using \eqref{EQ1}, \eqref{EQ2}, and \eqref{EQ3}. The resulting inequality being also valid for $(-\xi,-\eta),$ we obtain, by choosing successively $\xi=0$ and $\eta=0$, that
\begin{align}\label{R1}
&\frac{1}{\tau} \int_{\R^2} (\id-T)\cdot \xi\circ T\ f_0\, dx = \int_{\R^2} \left[ \Delta(A f + B g)\ \mathrm{div}(f\xi) + f\ \nabla(a f + b g)\cdot \xi \right]\, dx\,, \\ \label{R2}
&\frac{1}{\tau} \int_{\R^2} (\id-S)\cdot \eta\circ S\ g_0\, dx = \int_{\R^2} \left[ \Delta(f + g)\ \mathrm{div}(g\xi) + c\ f\ \nabla(f + g)\cdot \xi \right]\, dx\,.
\end{align}
Consider finally $\Xi\in C_0^\infty(\R^2)$ and take $\xi=\nabla \Xi$ in \eqref{R1}. Since 
$$
|\Xi(x)-\Xi(T(x))-\nabla\Xi(T(x))\cdot (x-T(x))| \le \frac{\|D^2\Xi\|_\infty\ |x-T(x)|^2}{2}
$$
for all $x\in\R^2$ by the mean value theorem, we find after multiplying the above relation by $f_0$ and thereafter integrating over $\R^2$ and using \eqref{new1} that
\begin{align*}
\left|\int_{\R^2}\left[\Xi(x)-\Xi(T(x))-\nabla\Xi(T(x))\cdot (x-T(x))\right]\ f_0(x)\, dx\right|&\leq\frac{\|D^2\Xi \|_\infty W^2_2(f,f_0)}{2}.
\end{align*}
Combining the above inequality with \eqref{R1} gives \eqref{A1}. 
The inequality \eqref{A2} next follows from \eqref{R2} in a similar way.
\end{proof}

In order to present the next result, we introduce first some notations.
Given a nonnegative and continuous function $h$ and $\delta>0,$ we define the open sets $\mathcal{P}_\delta^h$ and $\mathcal{P}^h$ by
$$
\mathcal{P}_\delta^h := \{x\in\R^2\ :\ h(x)>\delta\} \quad\text{ and }\quad \mathcal{P}^h := \bigcup_{\delta>0} \mathcal{P}_\delta^h\,.
$$

\begin{lemma}\label{le:h3} Given $(f_0,g_0)\in\cK_2$ and $\tau>0$, any minimizer $(f,g)$ of $\mathcal{F}_\tau$ in $\cK_2$ is such that $Af+Bg \in H^3_{\text{loc}}(\mathcal {P}^f)$ and $f+g\in H^3_{\text{loc}}(\mathcal {P}^g).$ Moreover, the functions $j_f$, $w_f$, $j_g$, and $w_g$ defined by  
\begin{equation}\label{QEf}
w_f := \left\{ 
\begin{array}{cl}
\sqrt{f}\ \left( -\nabla\Delta(A f+B g) + \nabla(a  f+b g) \right) & \;\;\text{a.e. in }\;\; \mathcal {P}^f\,, \\
0 & \;\;\text{a.e. in }\;\; \R^2\setminus\mathcal {P}^f\,,
\end{array}
\right. \,, \qquad j_f := \sqrt{f}\ w_f\,,
\end{equation}
and
\begin{equation}\label{QEg}
w_g:= \left\{ 
\begin{array}{cl}
\sqrt{g}\ \left( -\nabla\Delta( f+ g) + c\nabla(  f+ g)\right) & \;\;\text{a.e. in }\;\; \mathcal {P}^g\,, \\
0 & \;\;\text{a.e. in}\;\; \R^2\setminus\mathcal{P}^g\,,
\end{array}\right. \,, \qquad j_g := \sqrt{g}\ w_g\,,
\end{equation}
belong to $L_2(\R^2;\R^2)$ and satisfy
\begin{align}\label{jef1}
&\int_{\R^2} \left[ \Delta(A f+B g)\ \mathrm{div}(f\xi) + f\ \nabla(a  f+b g)\cdot \xi \right]\, dx = \int_{\R^2} j_f\cdot \xi\, dx, \\
\label{jef2}
&\int_{\R^2} \left[ \Delta( f+ g)\ \mathrm{div}(g\xi) + c\ g\ \nabla( f+ g)\cdot \xi \right]\, dx = \int_{\R^2} j_g\cdot\xi\, dx,
\end{align}
for all $\xi\in C_0^\infty(\R^2;\R^2).$ In addition, we have the following estimates:
\begin{equation}
\label{C1}
\tau\left\|w_f \right\|_{2}\leq W_2(f,f_0)\qquad\text{and} \qquad \tau\left\|w_g \right\|_{2}\leq W_2(g,g_0).
\end{equation}
\end{lemma}

\begin{proof}
Since $H^2(\R^2)$ is embedded in $C(\R^2)$, Lemma~\ref{Unimin} guarantees that $(f,g)\in C(\R^2;\R^2)$ so that $\mathcal{P}^f$ and $\mathcal{P}^g$ are indeed open subsets of $\R^2$. Next, recalling \eqref{R1}, we use once more the embedding of $H^2(\R^2)$ in $C(\R^2)$ as well as \eqref{new1} to obtain that, for $\xi\in C_0^\infty(\R^2;\R^2)$, 
\begin{align*}
& \left| \int_{\R^2} \left[ \Delta(A f+B g)\ \mathrm{div}(f\xi) + f\ \nabla(a  f+b g)\cdot \xi \right]\, dx \right| \\
\qquad &\le \frac{1}{\tau} \left( \int_{\R^2} |\id - T|^2\ f_0\, dx \right)^{1/2} \left( \int_{\R^2} |\xi\circ T|^2\ f_0\, dx \right)^{1/2} \\
\qquad &\le \frac{W_2(f,f_0)}{\tau} \left( \int_{\R^2} |\xi|^2 f\, dx\right)^{1/2} \le C \frac{W_2(f,f_0)}{\tau}\ \|f\|_{H^2}^{1/2}\ \|\xi\|_2.
\end{align*}
We may thus extend the functional
$$
\xi\longmapsto \int_{\R^2} \left[ \Delta(A f+B g)\ \mathrm{div}(f\xi) + f\ \nabla(a  f+b g)\cdot \xi \right]\, dx
$$
to a continuous linear functional on $L_2(\R^2;\R^2).$ Consequently, there exists a unique function $j_f\in L_2(\R^2;\R^2)$ having the property that
\begin{equation}\label{jef}
\int_{\R^2} \left[\Delta(A f+B g)\ \mathrm{div}(f\xi) + f\ \nabla(a  f+b g)\cdot \xi \right]\, dx = \int_{\R^2} j_f\cdot \xi\, dx 
\qquad\text{for all $\xi\in C_0^\infty(\R^2;\R^2).$}
\end{equation}
Since $(f,g)\in H^2(\R^2;\R^2)$ by Lemma~\ref{Unimin}, a density argument ensures that the relation \eqref{jef} is actually true for all $\xi\in H^1(\R^2;\R^2).$ 

Consider now $\delta>0$ and $\Xi\in C_0^\infty(\mathcal{P}^f_\delta;\R^2)$. Clearly $\Xi/f\in H^1(\R^2;\R^2)$ and we infer from \eqref{jef} with $\xi=\Xi/f$ that
\begin{equation}\label{jefrr}
\left| \int_{\mathcal{P}^f_\delta} \Delta(A f+B g)\ \mathrm{div}(\Xi)\, dx \right| \le \|\Xi\|_{L_2(\mathcal{P}^f_\delta)}\ \|\nabla(a f+b g)\|_2 + \|\Xi\|_{L_2(\mathcal{P}^f_\delta)}\ \frac{\|j_f\|_2}{\delta}\,.
\end{equation}
A duality argument then gives that $\Delta(Af+Bg)\in H^1(\mathcal{P}^f_\delta)$ for all $\delta>0.$ 
Consequently, we get  that $Af+Bg\in H^3_{\text{loc}}(\mathcal{P}^f)$ and together with \eqref{jef} we deduce
\begin{equation}\label{ds}
j_f = - f \nabla\Delta(A f+B g) + f\nabla(a  f+b g) \;\;\text{ a.e. in }\;\; \mathcal{P}^f\,.
\end{equation}

We next prove \eqref{C1}, adapting an argument from \cite[Proposition~2]{Ot98} and \cite[Corollary~2.3]{LMxx}. 
Let $\chi\in C_0^\infty(\R^2)$ be a non-negative function with $\|\chi\|_1=1$ and set $\chi_m(x) := m^2 \chi(mx)$ for $x\in\R^2$ and $m\ge 1.$ Since $H^2(\R^2)$ is embedded in $C(\R^2)$, we have
\begin{equation}
Y_m := \frac{1}{m} + \|\chi_m*f-f\|_{\infty}^{1/2} \longrightarrow 0 \;\;\text{ as }\;\; m\to\infty\,.\label{d11}
\end{equation}
Given $\vartheta\in C_0^\infty(\R^2;\R^2) $, the vector field $\vartheta/\sqrt{Y_m+\chi_m*f}$ belongs to $ C_0^\infty(\R^2;\R^2) $ too, 
and, by \eqref{new1}, \eqref{R1}, and \eqref{jef} with the choice $\xi=\vartheta/\sqrt{Y_m+\chi_m*f}$,
\begin{align*}
\left|\int_{\R^2} \frac{j_f\cdot \vartheta}{\sqrt{Y_m+\chi_m*f}} \, dx \right|&\leq \frac{W_2(f,f_0)}{\tau} \left(\int_{\R^2} |\vartheta|^2\ \frac{f}{Y_m+\chi_m*f}\, dx\right)^{1/2}\\
&\leq \frac{W_2(f,f_0)}{\tau} \left\|\frac{f}{Y_m+\chi_m*f}\right\|_\infty^{1/2}\ \|\vartheta\|_2.
\end{align*}
A duality argument then ensures that, for each $m\ge 1$, $j_f/\sqrt{Y_m+\chi_m*f}$ belongs to $L_2(\R^2;\R^2)$ with the estimate
$$
\left\| \frac{j_f}{\sqrt{Y_m+\chi_m*f}} \right\|_2 \le \frac{W_2(f,f_0)}{\tau} \left\|\frac{f}{Y_m+\chi_m*f}\right\|_\infty^{1/2}\,.
$$
Observing that 
$$
0 \le \frac{f}{Y_m+\chi_m*f} = \frac{f - \chi_m*f + \chi_m*f}{Y_m+\chi_m*f} \le \frac{\|f-\chi_m*f\|_\infty}{Y_m} + 1 \le 1+Y_m\,, 
$$
we actually have the estimate
\begin{equation}\label{jefy}
\left\|\frac{j_f }{\sqrt{Y_m+\chi_m*f}}\right\|_2\leq \frac{W_2(f,f_0)}{\tau}\ \left( 1 + Y_m \right).
\end{equation}
Several consequences can be drawn from \eqref{jefy}: first, since $Y_m\to 0$ as $m\to\infty$ by \eqref{d11}, 
the sequence $(j_f/\sqrt{Y_m + \chi_m*f})_m$ is bounded in $L_2(\R^2;\R^2)$ and there are thus a subsequence of $(j_f/\sqrt{Y_m + \chi_m*f})_m$ 
(not relabeled) and $w_f\in L_2(\R^2;\R^2)$ such that
\begin{equation}
\frac{j_f}{\sqrt{Y_m+\chi_m*f}} \rightharpoonup w_f \;\;\text{ in }\;\; L_2(\R^2;\R^2)\,. \label{d12}
\end{equation}
A simple consequence of \eqref{d11}, \eqref{jefy}, and \eqref{d12} is that 
\begin{equation}
\tau\ \|w_f\|_2 \leq W_2(f,f_0)\,. \label{d13}
\end{equation}
In addition, since $(\sqrt{Y_m+\chi_m*f})_m$ converges towards $\sqrt{f}$ uniformly on compact subsets of $\R^2$, we readily deduce from \eqref{d12} that
\begin{equation}
j_f = \sqrt{f}\ w_f \;\;\text{ a.e. in }\;\; \R^2\,. \label{d14}
\end{equation}
Next, since $f=0$ a.e. in $\R^2\setminus \mathcal{P}^f$, it follows from \eqref{jefy} that
\begin{align*}
\int_{\R^2\setminus \mathcal{P}^f} |j_f|^2\, dx &= \int_{\R^2\setminus \mathcal{P}^f} \frac{|j_f|^2 }{Y_m+\chi_m*f}\ \left( Y_m+\chi_m*f - f \right)\, dx \\
&\le \left\| \frac{j_f}{\sqrt{Y_m+\chi_m*f}} \right\|_2^2\ \left( Y_m + \|\chi_m*f - f \|_\infty \right) \\
&\le  \frac{W_2^2(f,f_0)}{\tau^2}\ (1+Y_m)^{3}\ Y_m \mathop{\longrightarrow}_{m\to \infty} 0\,,
\end{align*}
whence, additionally to \eqref{ds}, 
\begin{equation}
j_f = 0 \;\;\text{ a.e. in }\;\; \R^2\setminus \mathcal{P}^f\,. \label{d15}
\end{equation}
Finally, owing to \eqref{d12} and \eqref{d15}, we have
$$
\int_{\R^2\setminus \mathcal{P}^f} |w_f|^2\, dx = \lim_{m\to\infty} \int_{\R^2\setminus \mathcal{P}^f} w_f\cdot \frac{j_f}{\sqrt{Y_m+\chi_m*f}}\, dx = 0\,,
$$
and thus $w_f=0$ a.e. in $\R^2\setminus \mathcal{P}^f$. This completes the proof of Lemma~\ref{le:h3} for $f$. The statements \eqref{jef2} and \eqref{C1} for $g$ are proved by similar arguments.
\end{proof}

\section{Convergence of the time discretization}\label{sec:ctd}

We pick now $\tau>0$ and $(f_0,g_0)\in \cK_2$. For each integer $n\ge 1$,  we define $(f_\tau^{n+1},g_\tau^{n+1})\in \cK_2$ as a solution to the minimization problem
\begin{equation*}
\inf_{(u,v)\in\cK_2}\cF_\tau^n(u,v)\,,
\end{equation*}
where $(f_\tau^0,g_\tau^0) := (f_0,g_0)$ and
\begin{equation*}
\cF_\tau^n(u,v) := \frac{1}{2\tau}\ \left(W_2^2(u,f_\tau^n)+B\ W_2^2(v,g_\tau^n) \right) + \E(u,v)\,, \quad (u,v)\in\cK_2.
\end{equation*}
Recall that $(f_\tau^{n+1},g_\tau^{n+1})$ is well-defined and belongs to $H^2(\R^2;\R^2)$ for all $n\ge 1$ by Lemma~\ref{Unimin}. 
We next let $(f_\tau,g_\tau):[0,\infty)\times\R^2\to\cK_2$ be the function obtained  by the method of piecewise constant interpolation in $\cK_2$ as follows: $(f_\tau,g_\tau)(t) := (f_\tau^n,g_\tau^n)$ for all $t\in[n\tau,(n+1)\tau)$ and $n\in\N.$ 

The next lemma collects estimates which allow us to perform the limit $\tau\to 0$ and construct in this way a weak solution of \eqref{eq:problem}.

\begin{lemma}\label{L:unifestimates}
There is $C_3>0$ depending only on $A$, $B$, $a$, $b$, $c$, $f_0$, and $g_0$ such that, for all $T\geq0$ and $\tau\in(0,1)$, we have
\begin{align}
\label{a1}
(i)& \qquad \| f_\tau(T)\|_1 =\| g_\tau(T)\|_1 = 1,\\
\label{a2}
(ii)& \qquad\sum_{n=0}^\infty \left[W_2^2(f_\tau^{n+1},f_\tau^n )+W_2^2( g_\tau^{n+1}, g_\tau^n)\right]\leq C_3\ \tau,\\
\label{a3}
(iii)&\qquad \E(f_\tau(T),g_\tau(T))\leq \E(f_0,g_0),\\
\label{a4}
(iv)&\qquad \int_{\R^2} \left[ f_\tau(T,x) + g_\tau(T,x) \right]\ (1+|x|^2)\, dx \leq C_3\ (1+T),\\
\label{a5}
(v)&\qquad \int_\tau^{\max{\{T,\tau\}}} \left[ \|\Delta f_\tau(s)\|_2^2+\|\Delta g_\tau(s)\|_2^2\right]\, ds \leq C_3\ (1+T),\\
\label{a6}
(vi)&\qquad\int_\tau^\infty \left[ \left\| w_{f_\tau} \right\|_2^2 + \left\| w_{g_\tau} \right\|_2^2 \right]\,ds \leq C_3\,,
\end{align}
where
\begin{equation}\label{QEftau}
w_{f_\tau} := \left\{ 
\begin{array}{cl}
\sqrt{f_\tau}\ \left( -\nabla\Delta(A f_\tau + B g_\tau) + \nabla(a f_\tau +b g_\tau) \right) & \;\;\text{a.e. in }\;\; \mathcal {P}^{f_\tau}\,, \\
0 & \;\;\text{a.e. in }\;\; \R^2\setminus\mathcal {P}^{f_\tau}\,,
\end{array}
\right.
\end{equation}
and
\begin{equation}\label{QEgtau}
w_{g_\tau} := \left\{ 
\begin{array}{cl}
\sqrt{g_\tau}\ \left( -\nabla\Delta( f_\tau + g_\tau) + c\nabla(  f_\tau + g_\tau) \right) & \;\;\text{a.e. in }\;\; \mathcal {P}^{g_\tau}\,, \\
0 & \;\;\text{a.e. in}\;\; \R^2\setminus\mathcal{P}^{g_\tau}\,.
\end{array}\right. 
\end{equation}
\end{lemma}

\begin{proof}
The assertion \eqref{a1} follows from the fact that $(f_{\tau}^n,g_\tau^n) \in \cK_2$ for all $n\in\N$ and $\tau>0.$ 
We next observe that, since $\cF_\tau^n(f_\tau^n,g_\tau^n) \geq \cF_\tau^n(f_\tau^{n+1},g_\tau^{n+1})$ for all $n\in\N,$ we have
\begin{equation*}
\frac{1}{2\tau} \left[ W_2^2(f_\tau^{n+1},f_\tau^n) + B\ W_2^2(f_\tau^{n+1},g_\tau^n) \right] + \E(f_\tau^{n+1},g_\tau^{n+1}) \leq \E(f_\tau^n,g_\tau^n),
\end{equation*}
and therefore, for all $N\in\N,$
\begin{equation}\label{4a}
\frac{1}{2\tau} \sum_{n=0}^{N-1} \left[ W_2^2(f_\tau^{n+1},f_\tau^n) + B\ W_2^2(f_\tau^{n+1},g_\tau^n) \right] + \E(f_\tau^{N},g_\tau^{N}) \leq \E(f_0,g_0).
\end{equation}
Recalling that the functional $\E$ is bounded from below by Lemma~\ref{L:0}, we obtain \eqref{a2} after letting $N\to\infty$ in \eqref{4a}. 
Moreover, given $T\ge 0$, we choose $N\geq 1$ such  that $T\in[N\tau, (N+1)\tau)$ in \eqref{4a} and arrive at \eqref{a3}. 
Next, the bound \eqref{a4} follows readily from \eqref{a2} and the property $(f_0,g_0)\in\cK_2$ in a similar manner as \eqref{eq:6}.  

In order to deduce \eqref{a5}, we infer from Lemma~\ref{Unimin} that, for $n\in\N,$
\begin{equation*}
\begin{aligned}
&(A-B)\|\Delta f_\tau^{n+1}\|_2^2 + B\|\Delta (f_\tau^{n+1}+g_\tau^{n+1})\|_2^2 + (a-b)_+ \|\nabla f_\tau^{n+1}\|_2^2 + b \|\nabla (f_\tau^{n+1}+g_\tau^{n+1})\|_2^2\\
&\leq\frac{1}{\tau}\ \left[ H(f_\tau^n)-H(f_\tau^{n+1}) +B\left( H(g_\tau^n)-H(g_\tau^{n+1}) \right) \right] + (b-a)_+ \|\nabla f_\tau^{n+1}\|_2^2\,.
\end{aligned}
\end{equation*}
Summation from $n=0$ to $n=N-1$ yields
\begin{equation}\label{QQQ}
\begin{aligned}
\int_\tau^{(N+1)\tau} &\left[ (A-B)\|\Delta f_\tau(s)\|_2^2 + B\|\Delta (f_\tau+g_\tau)(s)\|_2^2 \right. \\
& \qquad \left. + (a-b)_+ \|\nabla f_\tau(s)\|_2^2 + b\ \|\nabla(f_\tau+g_\tau)(s) \|_2^2 \right]\, ds\\
\leq\ & \left[ H(f_0)-H(f_\tau^{N}) +B\left( H(g_0)-H(g_\tau^N) \right) \right]+(b-a)_+\int_\tau^{(N+1)\tau}\|\nabla f_\tau(s)\|_2^2\, ds.
\end{aligned}
\end{equation}
We now use Lemma~\ref{L:0}, Lemma~\ref{L:A2}, and the estimates \eqref{a3} and \eqref{a4}  to obtain
\begin{align*}
& \int_\tau^{(N+1)\tau} \left[ (A-B)\|\Delta f_\tau(s)\|_2^2 + B\|\Delta (f_\tau+g_\tau)(s)\|_2^2 \right]\, ds \\
& \le C_H + \int_{\R^2} f_0(x)\ (1+|x|^2)\, dx + \|f_0\|_2^2 + C_H +  \int_{\R^2} f_\tau^N(x)\ (1+|x|^2)\, dx \\
& + B \left( C_H + \int_{\R^2} g_0(x)\ (1+|x|^2)\, dx + \|g_0\|_2^2 + C_H +  \int_{\R^2} g_\tau^N(x)\ (1+|x|^2)\, dx \right) \\ 
& + \frac{4 (b-a)_+}{A-B}\ \int_\tau^{(N+1)\tau} \left( \E(f_\tau(s),g_\tau(s))+ C_1 \right)\, ds \le C\ (1+T)\,,
\end{align*}
for $T\in[N\tau,(N+1)\tau)$, hence \eqref{a5}. Finally, \eqref{a6} follows from \eqref{C1} and \eqref{a2}.
\end{proof} 

Using  uniform estimates from  Lemma \ref{L:unifestimates}, we now establish  the time equicontinuity of the family $(f_\tau,g_\tau)_\tau$. 
This step is one of the arguments needed to prove the compactness of  $(f_\tau,g_\tau)_\tau$.

\begin{lemma}
There exists a positive constant $C_4$ such that, for all $t\in [0,\infty)$, $s\in[0,\infty)$, and $\tau\in(0,1)$ we have
\begin{equation}\label{eq:comp}
\|f_\tau(t)-f_{\tau}(s)\|_{H^{-4}}+\|g_\tau(t)-g_{\tau}(s)\|_{H^{-4}}\leq C_4 \sqrt{|t-s|+\tau}.
\end{equation}
\end{lemma}

\begin{proof} Let $0\leq s<t$ with $s\in[\nu\tau,(\nu+1)\tau)$, $\nu\ge 0$, and $t\in[N\tau,(N+1)\tau)$, $N\ge\nu$, be given.
 By virtue of \eqref{A1}, \eqref{QEf}, and \eqref{jef1} we have for $\xi\in C_0^\infty(\R^2)$ and $n\ge 1$, 
$$
\left|\int_{\R^2} (f^{n}_\tau-f^{n-1}_\tau)\ \xi\, dx \right| \leq \tau \left|\int_{\R^2} j_{f_\tau^n} \nabla \xi\,dx \right| + \frac{\|D^2\xi\|_{\infty}\ W_2^2(f^{n}_\tau,f^{n-1}_\tau)}{2}\,.
$$
Using \eqref{C1}, \eqref{a1}, and H\"older's inequality, we obtain
\begin{align*}
\left|\int_{\R^2} (f^{n}_\tau-f^{n-1}_\tau)\ \xi\, dx \right| &\leq \tau \| f_\tau^n\|_1^{1/2} \left\| w_{f_\tau^n} \right\|_2\ \|\nabla\xi\|_\infty + \|D^2\xi\|_{\infty}\ W_2^2(f^{n}_\tau,f^{n-1}_\tau)\\
& \leq \|\nabla\xi\|_\infty\ W_2(f^{n}_\tau,f^{n-1}_\tau) + \|D^2\xi\|_{\infty}\ W_2^2(f^{n}_\tau,f^{n-1}_\tau)\,,
\end{align*}
whence, owing to the continuous embedding of $H^4(\R^2)$ in $W^2_\infty(\R^2)$, 
\begin{align*}
\left|\int_{\R^2} (f^{n}_\tau-f^{n-1}_\tau)\ \xi\, dx \right| & \leq C \left[ W_2(f^{n}_\tau,f^{n-1}_\tau) +W_2^2(f^{n}_\tau,f^{n-1}_\tau) \right]\ \|\xi\|_{H^4}\,.
\end{align*}
Therefore, by \eqref{a2}, we have
\begin{align*}
\left|\int_{\R^2} (f_\tau(t)-f_\tau(s))\ \xi\, dx \right|&\leq \sum_{n=\nu+1}^N \left| \int_{\R^2} (f^{n}_\tau-f^{n-1}_\tau)\ \xi\,dx\right| \\
&\leq C \|\xi\|_{H^4}\ \sum_{n=\nu+1}^N \left[ W_2(f^{n}_\tau,f^{n-1}_\tau) + W_2^2(f^{n}_\tau,f^{n-1}_\tau) \right]\\
&\leq C \|\xi\|_{H^4} \left[ \sqrt{N-\nu} \left(\sum_{n=\nu+1}^N W_2^2(f^{n}_\tau,f^{n-1}_\tau)\right)^{1/2} + C_3 \tau \right] \\
&\leq C \left[ \sqrt{(N-\nu)\tau} + \sqrt{\tau} \right] \|\xi\|_{H^4} \le C \sqrt{t-s+\tau}\ \|\xi\|_{H^4}\,,
\end{align*}
which yields \eqref{eq:comp} for $f_\tau$. A similar computation based on \eqref{A2}, \eqref{QEg}, and \eqref{jef2} gives \eqref{eq:comp} for $g_\tau$.
\end{proof}

We are now in a position to study the compactness properties of $(f_\tau,g_\tau)_\tau$ as $\tau\to 0$. 

\begin{lemma}\label{L:CO} There exist nonnegative functions $f$ and $g$ in $C([0,\infty), L_2(\R^2))$ and a subsequence  $(\tau_k)_{k\ge 1}$  which converges to zero such that, for all $t\ge 0$,
\begin{align}
\label{CO1}
&(f_{\tau_k}(t),g_{\tau_k}(t))\to (f(t), g(t))\qquad\text{in $L_2(\R^2;\R^2)$},\\
\label{CO2}
&\text{$(f_{\tau_k},g_{\tau_k})\to (f,g) $ \qquad in $L_2(0,t;H^1(\R^2;\R^2))$,}
\end{align}
and $(f(t),g(t))\in\cK_2.$ Moreover, we have $(f,g)\in L_2(0,t;H^2(\R^2;\R^2))$ and 
\begin{equation}
\text{$(f_{\tau_k},g_{\tau_k})\rightharpoonup (f,g) $ \qquad in $L_2(\delta,t;H^2(\R^2;\R^2))$} \label{CO3}
\end{equation}
for all $t>0$ and $\delta\in (0,t)$.
\end{lemma}

\begin{proof}
On the one hand, we remark that \eqref{eq:2} together with the estimate  \eqref{a3} and Lemma~\ref{L:0} imply that
\begin{equation}\label{E1}
\text{$(f_\tau)_{\tau\in(0,1)}$ and $(g_\tau)_{\tau\in(0,1)}$ are bounded in $L_\infty(0,\infty; H^1(\R^2;\R^2))$}.
\end{equation}
By interpolation, we have the inequality 
$$
\|h\|_2\leq \|h\|_{H^1}^{4/5}\|h\|_{H^{-4}}^{1/5}\,, \qquad h\in H^1(\R^2)\,,$$
which gives, together with \eqref{eq:comp} and \eqref{E1},
\begin{equation}\label{eq:comp2}
\|f_\tau(t)-f_{\tau}(s)\|_{2}+\|g_\tau(t)-g_{\tau}(s)\|_{2}\leq C(|t-s|+\tau)^{1/10}
\end{equation}
for all $\tau\in (0,1)$, $t\ge 0$, and $s\ge 0.$ 
On the other hand, for each $t\ge 0$, the sequence $(f_\tau(t),g_\tau(t))_{\tau\in(0,1)}$ lies in a  compact subset 
of $L_2(\R^2;\R^2)$ by \eqref{a4}, \eqref{E1}, and Lemma~\ref{le:comp}. 
Owing to these two properties, we can invoke \cite[Proposition~3.3.1]{AGS08} to conclude that there exists a  function $(f,g)\in C([0,\infty),L_2(\R^2;\R^2))$ and 
a subsequence $\tau_k\in (0,1) $, $\tau_k\to 0$, such that \eqref{CO1} holds true. In addition, we deduce from \eqref{CO1}, \eqref{E1}, and the Lebesgue dominated convergence theorem that 
\begin{equation}
(f_{\tau_k},g_{\tau_k}) \longrightarrow (f,g) \;\;\text{ in }\;\; L_2((0,T)\times\R^2) \;\;\text{ for all }\;\; T>0\,. \label{d20}
\end{equation}

We improve now this  convergence. 
Given $t\ge 1$ and $\delta\in (0,1),$ the estimates \eqref{a4}, \eqref{a5}, and \eqref{E1} ensure that
\begin{equation}\label{sera}
\int_\delta^t \left[ \|f_\tau(s)\|_{H^2}^2 + \|g_\tau(s)\|_{H^2}^2 \right]\ ds + \sup_{s\in (\delta,t)}\left\{ \int_{\R^2} (f_\tau+g_\tau)(s,x)\ (1+|x|^2)\ dx \right\} \le C\ (1+t)\,.
\end{equation}
By  Lemma~\ref{le:comp}, $H^2(\R^2)\cap L_1(\R^2,(1+|x|^2)\, dx)$ is compactly embedded in $H^1(\R^2),$ which in turn is continuously embedded in $L_2(\R^2),$ and we infer from \cite[Lemma~9]{Si87} that 
\begin{align*}
\text{$(f_{\tau_k},g_{\tau_k})\to (f,g) $ \qquad in $L_2(\delta,t;H^1(\R^2,\R^2))$,}
\end{align*}
which can be improved to \eqref{CO2} by using \eqref{E1}. 
Observing next that the right-hand side of \eqref{sera} does not depend on $\delta$, we realize that it follows from \eqref{sera} that, after possibly extracting a subsequence and using a diagonal process, we may assume that $(f,g)\in L_2(0,t;H^2(\R^2;\R^2))$ and that \eqref{CO3} holds true.

It remains to check that $(f(t),g(t))$ belongs to $\cK_2$ for all $t\ge 0$. Owing to \eqref{CO1} and \eqref{E1}, we readily obtain that $f(t)$ and $g(t)$ are both nonnegative and in $H^1(\R^2)$.
 In addition,   \eqref{CO1} and \eqref{sera} imply that $(f_{\tau_k}(t),g_{\tau_k}(t))_{k\ge 1}$ converges towards $(f(t),g(t))$ in $L_1(\R^2;\R^2)$ from which we deduce that $\|f(t)\|_1=\|g(t)\|_1=1$. 
Using once more \eqref{sera}, this convergence also guarantees that both $f(t)$ and $g(t)$ belong to $L_1(\R^2, (1+|x|^2)\, dx)$. Consequently, $(f(t),g(t))\in\cK_2$ for all $t\ge 0$ and the proof of Lemma~\ref{L:CO} is complete.
\end{proof}

\begin{proof}[Proof of Theorem~\ref{MT}]
Let us first check that the functions $(f,g)$ constructed in Lemma~\ref{L:CO} enjoy the regularity \textit{(i)} and \textit{(ii)} stated in Theorem~\ref{MT}. 
The boundedness and integrability properties \textit{(i)} follow at once from \eqref{E1} and \eqref{sera} by Lemma~\ref{L:CO}.
 We next use \eqref{CO1} to pass to the limit $k\to\infty$ in \eqref{eq:comp2} and obtain 
\begin{equation}\label{COC}
\|f(t)-f(s)\|_{2}+\|g(t)-g(s)\|_{2}\leq C |t-s|^{1/10}\qquad\text{for all $(s,t)\in [0,\infty)^2,$}
\end{equation}
which gives the assertion \textit{(ii)} of Theorem~\ref{MT}. 

We now identify the equations solved by $(f,g).$ 
For that purpose, we use  relations \eqref{A1} and \eqref{A2} to obtain, for $N\ge 1$, $t\in [N\tau,(N+1)\tau)$, and $\xi\in C_0^\infty(\R^2)$,
\begin{align}
\left| \int_{\R^2} (f_\tau(t)-f_0)\ \xi\, dx \right.
+& \left. \int_\tau^{(N+1)\tau} \int_{\R^2} \left[ \Delta(A f_\tau+B g_\tau)\ \mathrm{div}(f_\tau\nabla\xi) + f_\tau\ \nabla(a f_\tau+ b g_\tau)\cdot  \nabla\xi \right]\, dx\, ds\right| \nonumber\\
&\leq \ \frac{\|D^2\xi\|_{\infty}}{2}\ \sum_{n=1}^N W_2^2(f_\tau^n,f_\tau^{n-1}) \label{A11}
\end{align}
and
\begin{align}
\left| \int_{\R^2} (g_\tau(t)-g_0)\ \xi\, dx \right. +& \left.\int_\tau^{(N+1)\tau} \int_{\R^2}  \left[ \Delta( f_\tau+ g_\tau)\ \mathrm{div}( g_\tau \nabla\xi) + c\ g_\tau\ \nabla( f_\tau+  g_\tau)\cdot \nabla\xi \right]\, dx\,ds \right| \nonumber\\
&\leq  \frac{\|D^2\xi\|_{\infty}}{2}\ \sum_{n=1}^N W_2^2(g_\tau^n, g_\tau^{n-1})\,. \label{A22}
\end{align}

Let now $t>0$ be fixed. 
Before passing to the limit $\tau\to 0$ in \eqref{A11} and \eqref{A22}, let us point out that, owing to \eqref{E1} and \eqref{sera}, 
we have for all  integers $\nu\ge 1$ and $\tau>0$ with $(\nu+1)\tau\leq  t+1$,
\begin{align}
& \left| \int_{\nu\tau}^{(\nu+1)\tau} \int_{\R^2} \left[ \Delta(A f_\tau+B g_\tau)\ \mathrm{div}(f_\tau\nabla\xi) + f_\tau\ \nabla(a f_\tau+ b g_\tau)\cdot  \nabla\xi \right]\, dx\, ds \right| \nonumber\\
\le & \int_{\nu\tau}^{(\nu+1)\tau} \left[ \left( A \|\Delta f_\tau \|_{2} + B \|\Delta g_\tau \|_{2} \right) \|f_\tau\|_{H^1}\ \|\xi\|_{W_\infty^2} \right. \nonumber\\
& \hspace{2cm} + \left. \|f_\tau\|_2\ \|\nabla\xi\|_\infty\ \left( a \|\nabla f_\tau\|_{2} + b \|\nabla g_\tau\|_{2} \right) \right]\ ds \nonumber\\
\le & C (1+t)\ \sqrt{\tau}\ \|\xi\|_{W_\infty^2} \label{d21}
\end{align}
and
\begin{align}
& \left| \int_{\nu\tau}^{(\nu+1)\tau} \int_{\R^2}  \left[ \Delta( f_\tau+ g_\tau)\ \mathrm{div}( g_\tau \nabla\xi) + c\ g_\tau\ \nabla( f_\tau+  g_\tau)\cdot \nabla\xi \right]\, dx\,ds \right| \nonumber\\
\le & \int_{\nu\tau}^{(\nu+1)\tau} \left[ \left( \|\Delta f_\tau\|_2 + \|\Delta g_\tau\|_2 \right) \|g_\tau\|_{H^1}\ \|\xi\|_{W_\infty^2} \right. \nonumber\\
& \hspace{2cm} + \left. c\ \|g_\tau\|_2\ \|\nabla\xi\|_\infty\ \left( \|\nabla f_\tau\|_2 + \|\nabla g_\tau\|_2 \right) \right]\ ds \nonumber\\
\le & C (1+t)\ \sqrt{\tau}\ \|\xi\|_{W_\infty^2}\,. \label{d22}
\end{align}

We   fix  $\delta\in (0,t)$. For each $k\ge 1$, there are integers $N_k$ and $\nu_k$ such that $t\in [N_k\tau_k,(N_k+1)\tau_k)$ and $\delta\in [\nu_k \tau_k,(\nu_k+1)\tau_k)$. In virtue of \eqref{a2}, \eqref{A11}, and \eqref{d21} we obtain that, for $\xi\in C_0^\infty(\R^2)$, 
\begin{align}
& \left| \int_{\R^2} (f_{\tau_k}(t)-f_{\tau_k}(\delta))\ \xi\, dx \right. \nonumber\\
& \qquad + \left. \int_\delta^{t} \int_{\R^2} \left[ \Delta(A f_{\tau_k} + B g_{\tau_k})\ \mathrm{div}(f_{\tau_k}\nabla\xi) + f_{\tau_k} \nabla(a f_{\tau_k}+ b g_{\tau_k})\cdot \nabla\xi\right] dx\, ds \right| \nonumber\\
\leq &  \|D^2\xi\|_{\infty}\ \sum_{p=1}^\infty W_2^2(f_{\tau_k}^p,f_{\tau_k}^{p-1}) \nonumber\\
& +\left|\int_\delta^{(\nu_k+1){\tau_k}}\int_{\R^2} \left[ \Delta(A f_{\tau_k} + B g_{\tau_k})\ \mathrm{div}(f_{\tau_k}\nabla\xi) + f_{\tau_k}\ \nabla(a f_{\tau_k} + b g_{\tau_k})\cdot \nabla\xi \right] dx\, ds\right| \nonumber\\
&+\left|\int_{t}^{(N_k+1){\tau_k}}\int_{\R^2} \left[ \Delta(A f_{\tau_k} + B g_{\tau_k})\ \mathrm{div}(f_{\tau_k}\nabla\xi) + f_{\tau_k}\ \nabla(a f_{\tau_k}+ b g_{\tau_k})\cdot \nabla\xi\right] dx\, ds \right| \nonumber\\
\leq& C_3\ {\tau_k}\ \|D^2\xi\|_{\infty} + C (1+t)\ \sqrt{\tau_k}\ \|\xi\|_{W_\infty^2} \nonumber\\
\leq& C (1+t)\ \|\xi\|_{W^2_\infty}\ \sqrt{\tau_k}\,. \label{A111}
\end{align}

Let us now pass to the limit $\tau_k\to 0$ in \eqref{A111}. We note that the convergences \eqref{CO2} and \eqref{CO3} guarantee that
\begin{equation}
\begin{aligned}
&\Delta(A f_{\tau_k}+B g_{\tau_k})\ \nabla f_{\tau_k} \rightharpoonup \Delta(A f+B g)\ \nabla f \quad\text{in $L_1((\delta,t)\times\R^2;\R^2)$}, \\
& \Delta( Af_{\tau_k}+ Bg_{\tau_k})\ f_{\tau_k} \rightharpoonup \Delta(A f+B g)\ f \quad\text{in $L_1((\delta,t)\times\R^2)$},
\end{aligned}\label{Carol1}
\end{equation}
while \eqref{CO2} implies
\begin{equation}\label{Carol2}
f_{\tau_k}\ \nabla(a f_{\tau_k}+b g_{\tau_k}) \longrightarrow f\ \nabla (a f+b g) \quad\text{in $L_1((0,t)\times\R^2;\R^2)$}.
\end{equation}
We then let $k\to \infty$ in \eqref{A111} and use \eqref{CO1}, \eqref{Carol1}, and \eqref{Carol2} to conclude that
$$
\int_{\R^2} (f(t)-f(\delta))\ \xi\, dx + \int_\delta^{t} \int_{\R^2} \left[ \Delta(A f+B g)\ \mathrm{div}(f\nabla\xi) + f\ \nabla(a f+ b g)\cdot \nabla\xi \right] dx\, ds=0
$$
for all $\xi\in C^\infty_0(\R^2).$ By Lebesgue's dominated convergence theorem and \eqref{COC} we may let $\delta\to 0$ and thus obtain the first identity of \eqref{dem2}. 
The second identity of \eqref{dem2} follows in a similar way, starting from \eqref{A22} and \eqref{d22}.

\medskip

Let us now prove \eqref{dem3}. 
We fix $t>0$, $\delta\in (0,t)$ and take $k\ge 1$ sufficiently large so that $\tau_k\le \delta$ and  $n_k\ge 1$ such that $t\in [n_k\tau_k,(n_k+1)\tau_k)$. 
It follows from \eqref{QQQ} and \eqref{E1} that
\begin{equation}\label{energy}
\begin{aligned}
\int_{\delta}^{t} D_H(f_{\tau_k}(s),g_{\tau_k}(s))\, ds \le & \int_{\tau_k}^{(n_k+1)\tau_k} \left[ D_H(f_{\tau_k}(s),g_{\tau_k}(s)) + (b-a)\ \|\nabla f_{\tau_k}(s)\|_2^2 \right]\, ds \\
& + (a-b)\ \int_\delta^t \|\nabla f_{\tau_k}(s)\|_2^2\, ds \\
\le & H(f_0)-H(f_{\tau_k}(t)) + B \left( H(g_0)-H(g_{\tau_k}(t)) \right) \\
& + |b-a|\ \left( \int_{\tau_k}^\delta \|\nabla f_{\tau_k}(s)\|_2^2\, ds + \int_t^{(n_k+1)\tau_k} \|\nabla f_{\tau_k}(s)\|_2^2\, ds \right) \\
\le & H(f_0)-H(f_{\tau_k}(t)) + B \left( H(g_0)-H(g_{\tau_k}(t)) \right) + C\ |b-a\ \delta\,.
\end{aligned}
\end{equation}
Now, on the one hand, we infer from \eqref{CO2} and \eqref{CO3} that
\begin{align*}
\liminf_{k\to \infty} \int_{\delta}^{t} D_H(f_{\tau_k}(s),g_{\tau_k}(s))\, ds = & \liminf_{k\to \infty} \int_{\delta}^{t} \left[ D_H(f_{\tau_k}(s),g_{\tau_k}(s)) + (b-a)\ \|\nabla f_{\tau_k}\|_2^2 \right]\, ds \\
& + \lim_{k\to \infty} (a-b)\ \int_\delta^t \|\nabla f_{\tau_k}(s)\|_2^2\, ds \\
\ge & \int_{\delta}^{t} \left[ D_H(f(s),g(s)) + (b-a)\ \|\nabla f\|_2^2 \right]\, ds \\
& + (a-b)\ \int_\delta^t \|\nabla f(s)\|_2^2\, ds = \int_{\delta}^{t} D_H(f(s),g(s))\, ds\,. 
\end{align*}
On the other hand, it follows from \eqref{a4}, \eqref{CO1}, and \eqref{E1} by classical arguments that
$$
\lim_{k\to\infty} H(f_{\tau_k}(t)) + B H(g_{\tau_k}(t))  = H(f(t)) + B H(g(t))\,.
$$
see \cite{LMxx} for instance. 
Thanks to these two properties, we can pass to the limit $k\to\infty$ in \eqref{energy} and obtain
\begin{equation}\label{energy2}
\int_{\delta}^{t} D_H(f(s),g(s))\, ds \leq H(f_0) - H(f(t)) + B \left( H(g_0) - H(g(t)) \right) + C\ |b-a|\ \delta
\end{equation}
for all $\delta<t.$ 
By the monotone convergence theorem and the assertion $(i)$ of Theorem~\ref{MT} we may let $\delta\to 0$ in \eqref{energy2} and end up with \eqref{dem3}.

In order to obtain the last estimate \eqref{dem4}, we deduce from \eqref{C1} and \eqref{4a} that, if $t>\delta>0$ and $k$ is sufficiently large (so that $\tau_k<\delta$), then
\begin{equation}\label{raph1}
2 \E(f_0,g_0) \ge \int_\delta^t \left( \left\| w_{f_{\tau_k}} \right\|_2^2 + B \left\| w_{g_{\tau_k}} \right\|_2^2 \right) \, ds + 2 \E(f_{\tau_k}(t),g_{\tau_k}(t))\,,
\end{equation}
the functions $w_{f_{\tau_k}}$ and $w_{g_{\tau_k}}$ being defined in \eqref{QEftau} and \eqref{QEgtau}, respectively. Since $\E$ is bounded from below by Lemma~\ref{L:0}, we infer from \eqref{raph1} that $(w_{f_{\tau_k}})_k$ and $(w_{g_{\tau_k}})_k$ are bounded in $L_2((\delta,\infty)\times\R^2;\R^2)$ for all $\delta>0.$ Therefore, after possibly extracting a subsequence and using a diagonal process, we find vector fields $V_f$ and $V_g$ in $L_2((0,\infty)\times\R^2;\R^2)$ such that
\begin{equation}\label{raph}
\left( w_{f_{\tau_k}} , w_{g_{\tau_k}} \right) \rightharpoonup (V_f , V_g) \quad \text{in $L_2((\delta,\infty)\times\R^2;\R^2)$ for all $\delta>0.$}
\end{equation}
Owing to \eqref{CO1}, \eqref{E1}, and \eqref{raph}, we can first perform the  liminf $k\to \infty$ in \eqref{raph1}, then take the limit as $\delta\to 0$ with the help of the monotone convergence theorem in the resulting inequality, and thus arrive at
\begin{equation}\label{raph10}
2 \E(f_0,g_0) \ge \int_0^t \left( \|V_f\|_2^2 + B \|V_g\|_2^2 \right) \, ds + 2 \E(f(t),g(t)) \;\;\text{ for all }\;\; t\ge 0\,.
\end{equation}

It remains to identify the terms $V_f $ and $V_g.$ To this end, we remark first that \eqref{E1} ensures that $(\sqrt{f_{\tau_k}})_k$ and $(\sqrt{g_{\tau_k}})_k$ are bounded in $L_\infty(0,\infty; L_4(\R^2;\R^2))$, which implies, together with \eqref{raph1}, that the sequences $(j_{f_{\tau_k}})_k$ and $(j_{g_{\tau_k}})_k$ defined by $j_{f_{\tau_k}} = \sqrt{f_{\tau_k}}\ w_{f_{\tau_k}}$ and $j_{g_{\tau_k}} = \sqrt{g_{\tau_k}}\ w_{g_{\tau_k}}$, $k\ge 1$, are bounded in $L_2(\delta, \infty; L_{4/3}(\R^2;\R^2))$ for all $\delta>0.$
 Since $L_2(\delta, \infty; L_{4/3}(\R^2))$ is a reflexive space, there are vector fields $I_f$ and $I_g $ in $L_2(0, \infty; L_{4/3}(\R^2;\R^2))$ and a subsequence of $(\tau_k)_k$ (not relabeled) such that 
\begin{equation}\label{raph2}
\left( j_{f_{\tau_k}} , j_{g_{\tau_k}} \right) \rightharpoonup (I_f,I_g) \quad \text{in $L_2(\delta,\infty;L_{4/3}(\R^2;\R^2))$ for all $\delta>0.$}
\end{equation}
Combining \eqref{CO1}, \eqref{raph}, and \eqref{raph2} gives
\begin{equation}
I_f = \sqrt{f}\ V_ f \;\;\text{ and }\;\; I_g = \sqrt{g}\ V_g \;\;\text{ a.e. in }\;\; (0,\infty)\times\R^2\,. \label{raph3}
\end{equation}
 
 Consider now a test function $\Xi\in C_0^\infty((0,\infty)\times\R^2;\R^2)$. For each $k\ge 1$ and $n\ge 1$, we choose as test function
$$
\xi(x) = \int_{n\tau_k}^{(n+1)\tau_k} \Xi(s,x)\, ds\,, \qquad x\in\R^2\,,
$$
in \eqref{jef1} for $(f_{\tau_k}^n,g_{\tau_k}^n)$ and find, since  $(f_{\tau_k},g_{\tau_k})$ is constant on $[n\tau_k,(n+1)\tau_k)$:
\begin{align*}
\int_{n\tau_k}^{(n+1)\tau_k} \int_{\R^2} \left[ \Delta(A f_{\tau_k}+B g_{\tau_k})\ \mathrm{div}(f_{\tau_k} \Xi) \right. + &\left. f_{\tau_k}\ \nabla(a  f_{\tau_k} + b g_{\tau_k})\cdot \Xi \right]\, dx\, ds \\
= & \int_{n\tau_k}^{(n+1)\tau_k}  \int_{\R^2} j_{f_{\tau_k}}\cdot \Xi\, dx\, ds\,.
\end{align*}
Summing up the previous identity with respect to $n\ge 1$ gives
$$
\int_0^\infty \int_{\R^2} \left[ \Delta(A f_{\tau_k}+B g_{\tau_k})\ \mathrm{div}(f_{\tau_k} \Xi) + f_{\tau_k}\ \nabla(a  f_{\tau_k} + b g_{\tau_k})\cdot \Xi \right]\, dx\, ds = \int_0^\infty \int_{\R^2} j_{f_{\tau_k}}\cdot \Xi\, dx\, ds
$$
for $k$ large enough (such that $\mathrm{supp}(\Xi)\subset (\tau_k,\infty)\times\R^2$). Due to \eqref{CO2}, \eqref{CO3}, and \eqref{raph2}, we can pass to the limit as  $k\to\infty$ in the above equality and find
\begin{equation*} 
\int_0^\infty \int_{\R^2} \left[ \Delta(A f+B g)\ \mathrm{div}(f \Xi) + f\ \nabla(a  f+b g)\cdot \Xi \right]\, dx\, ds = \int_0^\infty \int_{\R^2} I_{f}\cdot \Xi\, dx\, ds\,,
\end{equation*}
that is, 
\begin{equation}\label{raph4}
I_f = - \nabla(f \Delta(Af+Bg)) + \Delta(A f+B g) \nabla f + f \nabla(a 
 f + b g)  \quad\text{in $\mathcal{D}'((0,\infty)\times\R^2;\R^2).$}
\end{equation}
A similar argument allows us to deduce from \eqref{jef2}, \eqref{CO2}, \eqref{CO3}, and \eqref{raph2} that 
\begin{equation}\label{raph5}
I_g = - \nabla(g \Delta(f+g)) + \Delta(f+g) \nabla g + c g \nabla(f+g)  \quad\text{in $\mathcal{D}'((0,\infty)\times\R^2;\R^2).$}
\end{equation}
Collecting \eqref{raph10}, \eqref{raph3}, \eqref{raph4}, and \eqref{raph5} gives the last assertion of Theorem~\ref{MT} and completes its proof. 

\end{proof}

\appendix
\section{Auxiliary results}

It is well-known that $H^1(\R^2)$ is not compactly embedded in $L_2(\R^2)$ due to the non-compactness of $\R^2$ but that compactness can be restored by an additional decay at infinity as in the following lemma:

\begin{lemma}\label{le:comp}
The spaces $H^1(\R^2)\cap L_1(\R^2,|x|^2\,dx)$ and $H^2(\R^2)\cap L_1(\R^2,|x|^2\, dx)$ are compactly embedded in $L_2(\R^2)$ and $H^1(\R^2)$, respectively.
\end{lemma}

\begin{proof}
Let  $(h_k)_{k\ge 1}$ be a bounded sequence in $H^1(\R^2)\cap L_1(\R^2,|x|^2\,dx)$. Without loss of generality, we may assume that there is a function $h\in H^1(\R^2)$ such that $h_k\rhp h$ in $H^1(\R^2).$ Furthermore, the Rellich theorem guarantees that $(h_{k|\D_N})_{k\ge 1}$ is relatively compact in $L_2(\D_N)$ for all integers $N\ge 1,$ where $\D_N$ is the open disc centered in zero and of radius $N$ and $h_{k|\D_N}$ the restriction of $h_k$ to $\D_N.$ We may then extract a subsequence, denoted again by $(h_k)_{k\ge 1},$ such that $(h_k)_{k\ge 1}$ converges (strongly) towards $h$ in $L_2(\D_N)$ for all $N\ge 1.$

First, for each $N\ge 1$, we have that
 \begin{align*}
\sup_{k\ge 1}\left\{ \int_{\R^2} |h_k(x)|\ |x|^2\, dx \right\} \geq \lim_{k\to\infty} \int_{\D_N} |h_k(x)|\ |x|^2\, dx = \int_{\D_N} |h(x)|\ |x|^2\, dx,
\end{align*}
so that $h\in L_1(\R^2,|x|^2\,dx).$ Next, we have
\begin{align*}
\|h_k-h\|_2^2 \leq& \int_{\D_N} |h_k(x)-h(x)|^2\, dx + \int_{\R^2\setminus\D_N} |(h_k-h)(x)|^2\, dx\\
\le & \int_{\D_N} |(h_k-h)(x)|^2\, dx + \frac{1}{N}\ \int_{\R^2\setminus\D_N} |(h_k-h)(x)|^2\ |x|\, dx\\
\le &\int_{\D_N} |(h_k-h)(x)|^2\, dx\\
& + \frac{1}{N}\ \left( \int_{\R^2} |(h_k-h)(x)|\ |x|^2\, dx \right)^{1/2} \|h_k-h\|_{3}^{3/2}.
\end{align*}
Since $H^1(\R^2)$ is continuously embedding in $L_3(\R^2)$ and the function $h$ belongs to $H^1(\R^2)$ and $L_1(\R^2,|x|^2\,dx),$ there exists a constant $C$ such that
$$
\|h_k-h\|_2^2\leq \int_{\D_N}|(h_k-h)(x)|^2\, dx+\frac{C}{N}\qquad\text{for all $k\geq1$ and $N\ge 1$.}
$$
Letting first $k\to\infty$ and then $N\to\infty$, we conclude that $(h_k)_{k\ge 1}$ converges towards $h$ in $L_2(\R^2)$ and thus that $H^1(\R^2)\cap L_1(\R^2,|x|^2\,dx)$ is compactly embedded in $L_2(\R^2)$.
  
Consider now a bounded sequence $(h_k)_{k\geq 1}$ in  $H^2(\R^2)\cap L_1(\R^2,|x|^2\, dx)$. Owing to the previous result, there exist a subsequence, denoted again by $(h_k)_{k\geq 1}$, and a function $h\in H^2(\R^2)\cap L_1(\R^2,|x|^2\,dx)$ such that $(h_k)_{k\ge 1}$ converges towards $h$ strongly in $L_2(\R^2)$ and weakly in $H^2(\R^2).$ Since
$$
\|w\|_{H^1} \le C\ \|w\|_{H^2}^{1/2}\ \|w\|_2^{1/2}\,, \qquad w\in H^2(\R^2)\,,
$$
a simple interpolation argument then gives that $(h_k)_{k\ge 1}$ converges towards $h$ strongly in $H^1(\R^2)$ and completes the proof.
\end{proof}

The next result was used in the identification of the Euler-Lagrange equation for the minimizers of the minimization problem \eqref{Zmin}.

\begin{lemma}\label{L:2} 
 Consider $h\in H^1(\R^2)$, $\zeta=(\zeta_1,\zeta_2)\in C_0^\infty(\R^2;\R^2)$ and, for $\e$ small enough, 
define $\zeta_\e := \id + \e\zeta$, $ j_\e:= \mathrm{det}(D\zeta_\e)$, and $h_\e := (h\circ \zeta_\e^{-1})/(j_\e\circ\zeta_\e^{-1})$. 
Then, there is $\e_\zeta>0$ such that $((h_\e-h)/\e)_{\e\in (0,\e_\zeta)}$ is bounded in $L_2(\R^2)$.
\end{lemma}

\begin{proof}
Let us first consider the case $h\in C_0^\infty(\R^2)$. Then, for $\e$ small enough,
\begin{align*}
\|h_\e - h \|_2^2  = & \int_{\R^2} \left| h_\e\circ \zeta_\e - h\circ\zeta_\e \right|^2\ j_\e\ dx = \int_{\R^2} \left| \frac{h}{j_\e} - h\circ\zeta_\e \right|^2\ j_\e\ dx \\ 
\le & 2\ \int_{\R^2} |h|^2\ \frac{(1-j_\e)^2}{j_\e}\ dx + 2\ \int_{\R^2} \left|  h - h\circ\zeta_\e \right|^2\ j_\e\ dx\,.
\end{align*}
Recalling that $j_\e = 1 + \e\ \mathrm{div} \zeta + o(\e)$, we have
\begin{equation}
\frac{1}{2} \le j_\e \le 2 \;\;\mbox{ and }\;\; \left| 1-j_\e \right|\le 2\ \|\zeta\|_{W^1_\infty}\ \e \label{newA}
\end{equation}
for $\e$ small enough, and we realize that
$$
2\ \int_{\R^2} |h|^2\ \frac{(1-j_\e)^2}{j_\e}\ dx \le 16\ \|\zeta\|_{W^1_\infty}^2\ \|h\|_2^2\ \e^2\,.
$$
Next, using once more \eqref{newA},
\begin{align*}
2\ \int_{\R^2} \left|  h - h\circ\zeta_\e \right|^2\ j_\e\ dx \le & 4\e^2\ \int_{\R^2} \left( \int_0^1 \left| \nabla h(x+\e s \zeta(x))\cdot \zeta(x) \right|\ ds\, \right)^2\, dx\\
 \le & 4\e^2\ \int_{\R^2} \left| \zeta(x) \right|^2\  \int_0^1 |\nabla (h\circ\zeta_{s\e})(x)|^2\, ds\, dx \\ 
\le & 4\e^2\ \|\zeta\|_\infty^2\ \int_0^1 \int_{\R^2} \frac{|\nabla h|^2}{j_{\e s}\circ\zeta_{\e s}^{-1}}\, dy\, ds \\
\le & 8\e^2\ \|\zeta\|_\infty^2\ \|\nabla h\|_2^2
\end{align*}
for sufficiently small $\e$. 
Combining the above inequalities gives the claimed boundedness of $(h_\e-h)/\e$ in $L_2(\R^2)$ for $\e$ small enough.

The general case $h\in H^1(\R^2)$ next follows by a density argument.
\end{proof}

We finally recall some well-known estimates for the functional $H$ defined in \eqref{eq:H}, see, e.g., \cite[Lemma~A.1]{LMxx}.

\begin{lemma}\label{L:A2} Let $h$ be a nonnegative function in $L_1(\R^2,(1+x^2)\, dx)\cap L_2(\R^2).$
Then $h\ln h\in L_1(\R^2)$ and there exists a positive constant $C_H$ such that
\begin{align}
\label{A:1}
\int_{\R^2} h(x)\ |\ln h(x)|\, dx & \leq C_H + \int_{\R^2} h(x) (1+|x|^2)\, dx+\|h\|_2^2,\\
\label{A:2}
H(h) & \geq -C_H - \int_{\R^2} h(x) (1+|x|^2)\, dx.
\end{align}
\end{lemma}

\bibliographystyle{abbrv}
\bibliography{LM_Mixed}

 \end{document}